\RequirePackage{etoolbox}
\csdef{input@path}{{style/}{graphics/}}
\documentclass[aap,MSNbibl,nameyear,seceqn,dvips]{arximspdf}
\usepackage{subeqn}
\usepackage{url,breakurl}

\doi{10.1214/14-AAP1009} 
\volume{25}
\issue{2}
\pubyear{2015}
\firstpage{714}
\lastpage{752}

\makeatletter
\newcommand{\eqref}[1]{(\ref{#1})}
\newtheorem{thmm}{Theorem}[section]
\newtheorem{cor}[thmm]{Corollary}
\newtheorem{lem}[thmm]{Lemma}
\newtheorem{prop}[thmm]{Proposition}
\newproclaim{defn}[thmm]{Definition}
\newproclaim{cond}[thmm]{Condition}
\newproclaim{ass}[thmm]{Assumption}
\newproclaim{rem}[thmm]{Remark}
\newproclaim{example}[thmm]{Example}

\newcommand{\wh}[1]{{\widehat{#1}}}

\newcommand{\RR}{\mathbb{R}}
\newcommand{\CC}{\mathbb{C}}
\newcommand{\DD}{\mathbb{D}}

\newcommand{\PP}{\mathbb{P}}
\newcommand{\FF}{\mathbb{F}}
\newcommand{\NN}{\mathbb{N}}

\newcommand{\cF}{\mathcal{F}}

\newcommand{\cB}{\mathcal{B}}
\newcommand{\cU}{\mathcal{U}}

\newcommand{\cY}{\mathcal{Y}}
\newcommand{\cR}{\mathcal{R}}

\newcommand{\pd}[2]{\frac{\partial #1}{\partial #2}}

\newcommand{\diag}{\operatorname{diag}}

\newcommand{\tr}{\operatorname{tr}}

\newcommand{\rank}{\operatorname{rank}}

\renewcommand{\Re}{\operatorname{Re}}
\renewcommand{\Im}{\operatorname{Im}}

\makeatother

\begin{document}
\begin{frontmatter}

\title{Exponential moments of affine processes}
\runtitle{Exponential moments of affine processes}

\begin{aug}
\author[A]{\fnms{Martin}~\snm{Keller-Ressel}\corref{}\ead[label=e1]{mkeller@math.tu-berlin.de}}
\and
\author[B]{\fnms{Eberhard}~\snm{Mayerhofer}\thanksref{T2}\ead[label=e2]{eberhard.mayerhofer@gmail.com}}
\thankstext{T2}{Supported by PITN-GA-2009-237984, PITN-GA-2009-237984,
ERC (278295) and SFI (08/SRC/FMC1389).}
\affiliation{TU Berlin and Dublin City University}

\address[A]{Fakult\"{a}t II\\
Institut f\"{u}r Mathematik, MA-705\\
TU Berlin\\
Stra{\ss}e des 17. Juni 136\\
D-10623 Berlin\\
Germany\\
\printead{e1}}
\address[B]{School of Mathematical Sciences\\
Dublin City University\\
Glasnevin, Dublin 9\\
Republic of Ireland\\
\printead{e2}}
\runauthor{M. Keller-Ressel and E. Mayerhofer}
\end{aug}

\received{\smonth{4} \syear{2012}}
\revised{\smonth{1} \syear{2014}}

%
\begin{abstract}We investigate the maximal domain of the moment
generating function of affine processes in the sense of Duffie,
Filipovi\'c and Schachermayer [\textit{Ann. Appl. Probab.} \textbf{13}
(2003) 984--1053], and we show the validity of the
affine transform formula that connects exponential moments with the
solution of a generalized Riccati differential equation. Our result
extends and unifies those preceding it (e.g., Glasserman and Kim
[\textit{Math. Finance} \textbf{20} (2010) 1--33],
Filipovi\'c and Mayerhofer [\textit{Radon Ser. Comput. Appl.
Math.} \textbf{8} (2009) 1--40] and Kallsen and Muhle-Karbe
[\textit{Stochastic Process Appl.} \textbf{120} (2010) 163--181])
in that it allows processes with very general jump behavior,
applies to any convex state space and provides both sufficient and
necessary conditions for finiteness of exponential moments.
\end{abstract}

%
\begin{keyword}[class=AMS]
\kwd[Primary ]{60J25}
\kwd[; secondary ]{91B28}
\end{keyword}

\begin{keyword}
\kwd{Affine process}
\kwd{exponential moment}
\kwd{Riccati equation}
\kwd{financial modeling}
\end{keyword}

\end{frontmatter}

\section{Introduction}\label{sec1}
This article investigates the maximal domain of the moment
generating function of an affine process. An affine process is a
time-homogeneous Markov processes $X$ on a finite-dimensional state
space $D\subset\mathbb R^d$ whose characteristic function has the
following property: There exist a complex-valued function $\phi$
and a $\mathbb C^d$-valued function $\psi$ such that
%
%
\begin{equation}
\label{ap} \Phi(t,u,x):=\mathbb E\bigl[e^{\langle{u},{X_t}\rangle}\mid X_0=x
\bigr]=e^{\phi(t,u)+
\langle{\psi(t,u)},{x}\rangle},
\end{equation}
for all $u\in i\mathbb R^d$, $t\geq0$ and $x\in D$. This so-called
\textit{affine property} implies that the PDE
\[
\frac{\partial}{\partial t} \Phi(t,u,x)=\mathcal A\Phi(t,u,x), \qquad \Phi (0,u,x)=\exp\bigl(
\langle{u},{x}\rangle\bigr),
\]
where $\mathcal A$ denotes the infinitesimal generator of $X$, can
be reduced to a system of nonlinear ODEs, commonly referred to as
\textit{generalized Riccati differential equations}, which are of the
form
\begin{subequations}\label{eq_Ric_intro}
%
%
\begin{eqnarray}
\frac{\partial}{\partial t} \phi(t,u)&=&F\bigl(\psi(t,u)\bigr),\qquad \phi(0,u)=0,
\\
\frac{\partial}{\partial t} \psi(t,u)&=&R\bigl(\psi(t,u)\bigr),\qquad \psi(0,u)=u.
\end{eqnarray}
\end{subequations}
A natural and important question is whether formula \eqref{ap} and
the generalized Riccati system \eqref{eq_Ric_intro} can be extended
to real exponential moments $(u \in\RR^d)$ or complex exponential
moments $(u \in\CC^d)$. One might expect that if $F$ and $R$ can be
suitably extended, for example, by analytic extension, then the
exponential moment $\mathbb{E}^{x}[e^{\langle{u},{X_T}\rangle}]$ is
finite if and only
if a solution to the extended Riccati system exists up to time $T$,
and that in this case also \eqref{ap} remains valid. A statement of
this type is usually referred to as \emph{affine transform formula}.
Showing such a formula in full generality is far from
trivial---difficulties include the fact that analytic extension of $F$
and $R$
may not be possible, that solutions of the extended Riccati
equations might not be unique and that the differentiability of $t
\mapsto\phi(t,u)$ and $t \mapsto\psi(t,u)$ is not obvious from
\eqref{ap}. The latter problem of showing that differentiability of
$\phi$ and $\psi$ can be concluded from the definition of an affine
processes is known as the \emph{regularity problem} for affine
processes; cf. Duffie, Filipovi{\'c} and
Schachermayer (\citeyear{Duffie2003}), \citet{KST2009},
\citet{Cuchiero2011}.

Several articles have been concerned with showing the affine transform
formula under different conditions on the process $X$ or the state
space $D$. In particular we mention the following contributions:

\begin{itemize}
\item\citet{Glasserman2009} show the affine transform formula for real
moments of
affine diffusion processes on $D = \mathbb{R}_{\geq0}^m \times\RR
^n$ under a
mean-reversion condition;
\item\citet{Filipovic2009} show the affine transform formula for real
and complex moments of
affine diffusion processes on $D = \mathbb{R}_{\geq0}^m \times\RR^n$;
\item\citet{Kallsen2008a} show that for affine semi-martingales on $D
= \mathbb{R}_{\geq0}^m \times\RR^n$ existence of a solution to the extended
Riccati system on $[0,T]$ implies the validity of the affine transform
formula for real moments under a mild condition on the jump-measures;
\item\citet{Spreij2010} show an affine transform formula for
affine processes whose jump measures possess exponential moments
of all orders and where the state space is a convex subset of $\RR^d$;
\item in the context of a stock price model with stochastic interest
rates and possibility of default, \citet{Cheridito2011} show an affine
transform formula for affine processes with killing when the jump
measures possess exponential moments of all orders.
\end{itemize}

In this article we generalize and unify most of these results. In
particular we remove the condition that all exponential moments of the
jump measures must exist, which is typically not fulfilled in
applications; see the discussion in Section~\ref{Sec:Examples}.
Moreover we show that the existence of a minimal solution to the
extended Riccati system is necessary and sufficient for the exponential
transform formula to hold, while \citet{Kallsen2008a} covers only
sufficiency. Finally our results apply to very general types of state
spaces: The results on real exponential moments hold for affine
processes on an arbitrary convex state space, and the results on
complex exponential moments apply to affine processes on $D=\mathbb
R_+^m\times\mathbb R^n$ and on $D=S_d^+$ (the positive semidefinite
$d\times d$ matrices). These two state spaces [see Duffie, Filipovi{\'c} and
Schachermayer (\citeyear{Duffie2003})
and \citet{cfmt}] are of
particular interest both from the theoretic viewpoint and from the
applied one.
The outline of this paper is as follows: In Section~\ref{sec:
definitions} we present general definitions, some useful notation
and our main results:

\begin{itemize}
\item Theorem~\ref{thmm:main_results} proves the affine transform
formula in terms of \textit{minimal solutions}
to the so-called \textit{extended Riccati system}, which comes from
considering \eqref{eq_Ric_intro} in the real domain. Here we only
require the state space to be closed, convex and with nonempty
interior. The proof of Theorem~\ref{thmm:main_results} is provided in
Section~\ref{sec: real proofs}.
\item Theorem~\ref{thmm:main_complex} extends the validity of the
affine property \eqref{ap} to complex moments $u=p+i z$, where $z\in
\mathbb R^d$. This extension succeeds
under the premise that the $p$th real moment is finite, or
equivalently, that the extended Riccati equations are solvable until
time $T$. The result holds for the state space $\mathbb{R}_{\geq0}^m
\times
\RR^n$ and---under some mild additional conditions---for the state
space $S_d^+$. For the proof of Theorem~\ref{thmm:main_complex}, see
Section~\ref{sec: complex}.
\end{itemize}
In Section~\ref{sec:applications} several applications of our
results to mathematical finance are outlined. Finally, Sections~\ref
{sec: real proofs} and \ref{sec: complex} contain the proofs of
our main results for real moments and complex moments respectively.

\section{Definitions and main results}\label{sec: definitions}
\subsection{Affine processes}
Let $(\Omega, \cF, \FF)$ be a filtered space, with $\FF= (\cF_t)_{t
\ge0}$ a right-continuous filtration. We endow $\RR^d, (d \ge1)$
with an inner product $\langle{\cdot},{\cdot}\rangle$ and let $D$ be a nonempty
convex subset of $\RR^d$, which will act as the state space of the
stochastic process $X$ we are about to define. The state space $D$
has a measurable structure given by its Borel $\sigma$-algebra
$\cB(D)$, and without loss of generality (see the explanation
after Definition~\ref{Def:affine_process}), we may assume that $D$
contains $0$ and that the linear span of $D$ is the full space
$\RR^d$. Under this assumption it follows in particular that the
interior $D^\circ$ of $D$ is nonempty. Associated to $D$ is the
set
%
%
\begin{equation}
\label{Eq:U_def} \cU= \bigl\{u \in\CC^d\dvtx x \mapsto e^{\langle{x},{u}\rangle}
\mbox{is bounded on } D\bigr\}.
\end{equation}
Finally let $(\PP^x)_{x \in D}$ be a
family of probability measures on the filtered space $(\Omega, \cF,
\FF)$ and assume that $\cF$ is complete with respect to $(\PP^x)_{x
\in
D}$ in the sense of \citet{Blumenthal1968}, Chapter~I.5.

Let $X$ be a c\`adl\`ag\setcounter{footnote}{1}\footnote{For convex state spaces, affine
processes have c\`adl\`ag modifications, see Remark~\ref{remcadlag}
below.} $\FF$-adapted time-homogeneous conservative Markov process with
state space $D$. More precisely, writing
%
%
\begin{equation}
\label{Eq:transition} p_t(x,A) = \PP^x(X_t \in A)\qquad
\bigl(t \ge0, x \in D, A \in\cB(D)\bigr)
\end{equation}
for the transition kernel of $X$, $p_t(x,A)$ satisfies the following:
\begin{longlist}[(a)]
\item[(a)]$x \mapsto p_t(x,A)$ is $\cB(D)$-measurable for all $t \ge0, A
\in\cB(A)$,
\item[(b)]$p_t(x,D) = 1$ for all $t \ge0, x \in D$,
\item[(c)]$p_0(x,\{x\}) = 1$ for all $x \in D$ and
\item[(d)] the Chapman--Kolmogorov equation
\[
p_{t+s}(x,A) = \int_D p_t(y,A)
p_s(x,dy)
\]
holds for every $t,s \ge0$ and $(x,A) \in D \times\cB(D)$.
\end{longlist}

%
\begin{rem}
Since $X$ is c\`adl\`ag, the law of $X$ under $\PP^x$ is a probability
measure on the Skorokhod space of c\`adl\`ag paths $\DD(\mathbb
{R}_{\geq0}, \RR^d)$,
for each $x \in D$. There will be no loss of generality by directly
interpreting $\PP^x$ as a measure on this path space.
\end{rem}

%
\begin{defn}[(Affine process)]\label{Def:affine_process}
The process $X$ is called \emph{affine} with state space $D$, if its
transition kernel $p_t(x,A)$ satisfies the following:
\begin{longlist}[(i)]
\item[(i)] it is stochastically continuous, that is, $\lim_{s \to t}p_s(x,\cdot)
= p_t(x,\cdot)$ weakly for all $t \ge0, x \in D$, and
\item[(ii)] there exist functions $\phi\dvtx  \mathbb{R}_{\geq0}\times\cU\to
\CC$ and
$\psi\dvtx  \mathbb{R}_{\geq0}\times\cU\to\CC^d$ such that
%
%
\begin{equation}
\label{Eq:affine_property} \int_D e^{\langle{u},{z}\rangle} p_t(x,d
\xi) = \exp\bigl(\phi(t,u) + \bigl\langle{x},{\psi(t,u)}\bigr\rangle\bigr)
\end{equation}
for all $t \ge0, x \in D$ and $u \in\cU$.
\end{longlist}
\end{defn}

%
\begin{rem}We explain why it is no loss of generality to assume that
$D$ contains $0$ and linearly spans the whole space $\RR^d$:
For an arbitrary nonempty convex subset $D$ of $\RR^d$, let
$\mathrm{aff}(D)$ be the smallest affine subspace of $\RR^d$ that
contains $D$, and let $(x_0, x_1, \ldots, x_k)$ be an affine basis
of $\mathrm{aff}(D)$ such that $x_0 \in D$. Let $h\dvtx \mathrm{aff}(d)
\to\RR^k\dvtx x \mapsto A^\top(x-x_0)$ be the projection to canonical
affine coordinates, that is, $h(x_0) = 0$ and $h(x_i) = e_i$ for each $i
\in\{1, \ldots, k\}$. Set ${\widetilde{D}} = h(D) \subset\RR^k$ and
${\widetilde{X}} = h(X)$. Then ${\widetilde{D}}$ is convex, contains
$0$ and linearly
spans $\RR^k$. It is easily verified that ${\widetilde{X}}$ is again an
affine process with
%
%
\begin{eqnarray}
{\widetilde{\phi}}(t,u) &=& \phi(t,Au) + \bigl\langle{x_0},{
\psi(t,Au) - u}\bigr\rangle,
\\
{\widetilde{\psi}}(t,u) &=& A^+\psi(t,Au),
\end{eqnarray}
where $A^+$ is the Pseudoinverse of $A$ (or any other $k \times
d$-matrix such that $A^+ A = \mathrm{id}_k$).
\end{rem}

The next result shows that an affine process is a semimartingale
with affine (differential) semimartingale characteristics.

%
\begin{thmm}[{[\citet{Cuchiero2011}]}]\label{Thm:semimartingale}
Let $X$ be an affine process with state space $D \subset\RR^d$. Then
for each $x \in D$, the process $X$ is a $\PP^x$-semimartingale with
semimartingale characteristics
\begin{subequations}\label{Eq:semimartingale_char}
%
%
\begin{eqnarray}
A_t &=& \int_0^t
a(X_{s-}) \,ds,
\\
B_t &=& \int_0^t
b(X_{s-}) \,ds,
\\
\nu(\omega,dt,d\xi) &=& K\bigl(X_{t-}(\omega),d\xi\bigr) \,dt,
\end{eqnarray}
\end{subequations}
where $a(x)$, $b(x)$ and $K(x,d\xi)$ are affine functions of the form
\begin{subequations}\label{Eq:characteristics_affine}
%
%
\begin{eqnarray}
a(x) &=& a + x_1 \alpha^1 + \cdots+ x_d
\alpha^d,
\\
b(x) &=& b + x_1 \beta^1 + \cdots+ x_d
\beta^d,
\\
K(x,d\xi) &=& m(d\xi) + x_1 \mu^1(d\xi) + \cdots+
x_d \mu^d(d\xi),
\end{eqnarray}
\end{subequations}
and for each $x \in D$ it holds that $a(x)$ is a positive semidefinite $d
\times d$ matrix, $b(x)$ is a $\RR^d$-vector and $K(x,d \xi)$ is a
Radon measure on $\RR^d$, satisfying
\[
\int_{\RR^d}\bigl(\Vert\xi\Vert^2 \wedge1
\bigr)K(x,d\xi) < \infty
\]
and $K(x,\{0\}) = 0$.
\end{thmm}

\begin{pf}
Follows from \citet{Cuchiero2011}, Theorems~1.4.8 and 1.5.4.
\end{pf}

%
\begin{rem}\label{remcadlag}
Note that several of the assumptions made at the beginning of the
section could be slightly weakened: Following \citet{CT13}
\emph{any} affine process (satisfying a mild regularity property on
$\phi,\psi$ which is automatically fulfilled for convex state spaces)
has a c\`adl\`ag modification; moreover the
$(\mathbb P^x)_{x \in D}$-completion of the filtration generated by an
affine process is
automatically right continuous. Note that it is unkown to this date, whether
all affine processes are Feller. Hence the proof of the c\`adl\`ag
modification in \citet{CT13} is not
an immediate consequence of the Feller property, but more involved.
\end{rem}

\subsection{Real moments of affine processes}

%
\begin{defn}\label{Def:Rx} Given an affine process $X$ and the
associated functions $(a(x),b(x),K(x,d\xi))$ in \eqref
{Eq:semimartingale_char}, define for each $x \in D$ the function $\cR
_x\dvtx \RR^d \to(-\infty, \infty]$ by
%
%
\begin{eqnarray}
\label{Eq:Rx_def} \cR_x(y)& =& \frac{1}{2}\bigl\langle{y},{a(x) y}
\bigr\rangle + \bigl\langle {b(x)},{y}\bigr\rangle
\nonumber
\\[-8pt]
\\[-8pt]
\nonumber
&&{} + \int_{\RR^d
\setminus\{0\}}{
\bigl(e^{\langle{\xi},{y}\rangle} - 1 - \bigl\langle{h(\xi)},{y}\bigr\rangle\bigr) K(x,d
\xi)},
\end{eqnarray}
where $h(\xi) = \mathbf{1}_{\{|\xi| \le1\}} \xi$.
\end{defn}

For each fixed $x \in D$, the function $\cR_x$ is a convex and lower
semi-continuous function\footnote{Lower semi-continuity follows from
Fatou's lemma applied to the integral with respect to $K(x,d\xi)$.}
that may take the value $+\infty$. As for any convex function, the
effective domain $\cY_x$ is the set of arguments for which $\cR_x$
takes finite values. Taking the intersection over all $x \in D$
leads to the following definition.

%
\begin{defn}\label{Def:D}
Given an affine process $X$ and the associated function $\cR_x$ as in
Definition~\ref{Def:Rx}, define
%
%
\begin{equation}
\label{eq: description cD} \cY= \bigcap_{x \in D} \biggl\{y \in
\RR^d\dvtx \int_{|\xi| \ge1} e^{\langle{y},{\xi}\rangle} K(x,d\xi) <
\infty\biggr\}.
\end{equation}
\end{defn}

As an intersection of convex sets, also $\cY$ is convex. Moreover,
$\cY$ contains $0$ and hence is nonempty, because $\cR_x(0) = 0$
for all $x \in D$.

Since the functions $a(x), b(x)$ and $K(x,d\xi)$ are affine in $x$,
we can decompose $\cR_x$ into $\cR_x(y) = F(y) + \langle
{R(y)},{x}\rangle$.
For arguments $y \in\cY$, the functions $F$ and $R$ are uniquely
specified, since $D$ contains $0$ and $d$ linearly independent
points.

%
\begin{prop} \label{Prop:FR} Let $X$ be an affine process with state
space $D$. Then there exist functions $F\dvtx \cY\to\RR$, $R\dvtx \cY\to
\RR
^d$ such that
\[
\cR_x(y) = F(y) + \bigl\langle{R(y)},{x}\bigr\rangle
\]
for all $x \in D$, $y \in\cY$. Let $(e_1, \ldots, e_d)$ be the
canonical basis vectors in $\RR^d$. Then we can write $F$
and $R_i(y):= \langle{R(y)},{e_i}\rangle$ as
\begin{subequations}\label{Eq:FR_def}
%
%
\begin{eqnarray}
F(y) &=& \frac{1}{2}\langle{u},{a y}\rangle + \langle{b},{y}\rangle
\nonumber
\\[-8pt]
\\[-8pt]
\nonumber
&&{}+ \int
_{\RR^d
\setminus\{0\}}{\bigl(e^{\langle{\xi},{y}\rangle} - 1 - \bigl\langle{h(\xi )},{y}
\bigr\rangle\bigr) m(d\xi)},
\\
R_i(y) &= &\frac{1}{2}\bigl\langle{y},{\alpha^i y}
\bigr\rangle + \bigl\langle {\beta^i},{y}\bigr\rangle
\nonumber
\\[-8pt]
\\[-8pt]
\nonumber
&&{}+ \int
_{\RR^d \setminus\{0\}}{\bigl(e^{\langle{\xi},{y}\rangle} - 1 - \bigl\langle{h(\xi)},{y}
\bigr\rangle\bigr) \mu^i(d\xi)},
\end{eqnarray}
with $h(\xi) = \mathbf{1}_{\{|\xi| \le1\}} \xi$.
\end{subequations}
\end{prop}

\begin{pf}The proof follows immediately from Definition~\ref{Def:Rx}
and Theorem~\ref{Thm:semimartingale}.
\end{pf}

%
\begin{rem}
Setting $x = 0$ in \eqref{Eq:Rx_def} yields that $F(y)$ is a convex
and lower semi-continuous function of L\'evy--Khintchine form. The
same is not necessarily true for $R_1, \ldots, R_d$, since the
matrices $\alpha^i$ may not be positive semidefinite, or the measures
$\mu^i$ may be signed measures.
\end{rem}

We use the functions $F(y)$ and $R(y)$ to set up a system of ODEs
associated to the affine process $X$. These equations play a key role
in our main result.

%
\begin{defn}[(Extended Riccati system)]
Let $X$ be an affine process and $F,R$ and $\cY$ be defined as in
Definition~\ref{Def:D} and Proposition~\ref{Prop:FR}. Let $T \ge0,
y \in\cY$ and let
\[
p\dvtx t \mapsto p(t,y), q\dvtx t \mapsto q(t,y)
\]
be $C^1$-functions mapping $[0,T]$ to $\RR$ (resp., $\cY$) that satisfy
\begin{subequations}\label{Eq:Eric}
%
%
\begin{eqnarray}
\label{Eq:Eric1} \frac{\partial}{\partial t} p(t,y) &=& F\bigl(q(t,y)\bigr), \qquad p(0,y) = 0,
\\
\label{Eq:Eric2} \frac{\partial}{\partial t} q(t,y) &=& R\bigl(q(t,y)\bigr), \qquad q(0,y) = y
\end{eqnarray}
\end{subequations}
for all $t \in[0,T]$. Then we call $(p,q)$ a solution (up to time $T$
and with starting point $y$) of the extended Riccati system associated
to $X$.
\end{defn}

It is important to note that in general the function $R$ is locally
Lipschitz continuous only on the interior of $\cY$, but may fail to
be Lipschitz continuous at the boundary of $\cY$. Hence solutions of
\eqref{Eq:Eric} reaching or starting at the boundary of $\cY$ may
not be unique. For this reason we add the following definition.

%
\begin{defn}[(Minimal solution)]\label{definminsol}
Let $X$ be an affine process, and let $(p,q)$ a solution up of $T$
starting at $y \in\cY$ to the associated extended Riccati system.
We call $(p,q)$ a \emph{minimal} solution, if for any other solution
$({\widetilde{p}},{\widetilde{q}})$ up to ${\widetilde{T}} \le T$
and starting at the same point
$q(0,y) = {\widetilde{q}}(0,y) = y$ it holds that
%
%
\begin{equation}
p(t,y) + \bigl\langle{q(t,y)},{x}\bigr\rangle \le{\widetilde{p}}(t,y) + \bigl
\langle{{\widetilde{q}}(t,y)},{x}\bigr\rangle
\end{equation}
for all $t \in[0,{\widetilde{T}}]$ and $x \in D$.
\end{defn}

%
\begin{rem}
By setting $q_x(t,y): = p(t,y) + \langle{q(t,y)},{x}\rangle$, the extended
Riccati system may be written in condensed form as
%
%
\begin{equation}
\label{Eq:Eric_condensed} \frac{\partial}{\partial t}q_x(t,y) = \cR_x
\bigl(q(t,y)\bigr),\qquad  q_x(0,y) = \langle y,x\rangle\qquad \forall x \in D.
\end{equation}
In this notation the minimality property can we written as
\[
q_x(t,y) \le{\widetilde{q}}_x(t,y)\qquad \forall x\in D, t
\in [0,{\widetilde{T}}], y \in\cY.
\]
\end{rem}

%
\begin{rem}\label{rem112}
The following properties are easy to see: If for a given starting
value $y \in\cY$ there is only one solution to the extended Riccati
system, then it is automatically a minimal solution. Also, if for a
given starting value a minimal solution $(p,q)$ exists up to time
$T$, it is automatically the unique minimal solution. Indeed, if
there were another minimal solution $({\widetilde{p}},{\widetilde
{q}})$, then
\[
p(t,y) + \bigl\langle{q(t,y)},{x}\bigr\rangle = {\widetilde{p}}(t,y) + \bigl
\langle {{\widetilde{q}}(t,y)},{x}\bigr\rangle
\]
for all $t \in[0,T]$, $x \in D$. Since $D$ contains $d$ linearly
independent points and $0$, it follows that $p = {\widetilde{p}}$ and
$q = {\widetilde{q}}$ in this case.
\end{rem}

We can now formulate our main results on the behavior of exponential
moments of affine processes.

%
\begin{thmm}[(Real moments of affine processes)]\label{thmm:main_results}
Let $X$ be an affine process on $D$, and let $T \ge0$.
\begin{longlist}[(a)]
\item[(a)]
Let $y \in\RR^d$, and suppose that $\mathbb
{E}^{x}[e^{\langle{y},{X_T}\rangle}] < \infty$ for some $x \in
D^\circ$. Then $y \in
\cY$ and
there exists a unique minimal solution $(p,q)$ up to time $T$ of the
extended Riccati system \eqref{Eq:Eric}, such that
%
%
\begin{equation}
\label{eq: transform formula} \mathbb{E}^{x}\bigl[e^{\langle{y},{X_t}\rangle}\bigr] = \exp
\bigl(p(t,y) + \bigl\langle {q(t,y)},{x}\bigr\rangle\bigr)
\end{equation}
holds for all $x \in D$, $t \in[0,T]$.
\item[(b)]
Let $y \in\cY$, and suppose that the extended
Riccati system \eqref{Eq:Eric} has solutions $({\widetilde{p}},
{\widetilde{q}})$ that
start at $y$
and exist up to $T$. Then $\mathbb{E}^{x}[e^{\langle{y},{X_T}\rangle
}] < \infty$ and
there exist unique minimal solutions $(p,q)$ up to time $T$ of the
extended Riccati system such that \eqref{eq: transform formula}
holds for all $x \in D$, $t \in[0,T]$.
\end{longlist}
\end{thmm}

%
\begin{rem}We emphasize that in point (b) of the theorem
$p = {\widetilde{p}}$ and $q = {\widetilde{q}}$ does not necessarily
hold, that is, the
candidate solutions $({\widetilde{p}},{\widetilde{q}})$ have to be
replaced by the
minimal solutions $(p,q)$ in order for \eqref{eq: transform formula}
to hold true.
\end{rem}

The following corollary is a conditional version of
Theorem~\ref{thmm:main_results} and thus extends the corresponding
result [\citet{Filipovic2009}, Theorem~3.3(iv)] for affine diffusions
on canonical state-spaces:

%
\begin{cor}
Suppose that the conditions of either
Theorem~\ref{thmm:main_results}\textup{(a)}
or~\textup{(b)} are
satisfied,
and let $(p,q)$ be the associated minimal
solutions of the Riccati system \eqref{Eq:Eric}. Then also $\mathbb
{E}^{x}[e^{\langle{q(T-t,y)},{X_t}\rangle}] < \infty$ and
\[
\mathbb{E}^{x}\bigl[e^{\langle{y},{X_T}\rangle}|\cF_t\bigr] = \exp
\bigl(p(T-t,y) + \bigl\langle q(T-t,y, X_t)\bigr\rangle\bigr)
\]
holds for all $x \in D$, $t \in[0,T]$.
\end{cor}

The next proposition provides a way to identify whether some solution
$({\widetilde{p}},{\widetilde{q}})$ of the extended Riccati system is
in fact the minimal
solution.

%
\begin{prop}\label{Prop:minimal_solutions}
Let $X$ be an affine process, and let $({\widetilde{p}},{\widetilde
{q}})$ be a solution
up to time $T \ge0$ of the extended Riccati system associated to $X$.
Each of the following
conditions is sufficient for $({\widetilde{p}},{\widetilde{q}})$ to
be the unique minimal
solution:
\begin{longlist}[(a)]
\item[(a)]$X$ is a diffusion process; 
\item[(b)]$\cY= \RR^d$; 
\item[(c)]$\cY$ is open; 
\item[(d)]${\widetilde{q}}(t,y) \in\cY^\circ$ for all $t \in[0,T)$.
\end{longlist}
\end{prop}

\begin{pf}
From Definition~\ref{Def:D} of $\cY$ it follows that $\mathrm{(a)}
\Rightarrow\mathrm{(b)} \Rightarrow\mathrm{(c)} \Rightarrow\mathrm{(d)}$, that is, it is
sufficient to show that (d) implies uniqueness of the solution
$({\widetilde{p}},{\widetilde{q}})$. But $R$ is locally Lipschitz on
$\cY^\circ$, such
that standard ODE results imply that $({\widetilde{p}},{\widetilde
{q}})$ is the unique
(and hence unique minimal) solution of the extended Riccati system
\eqref{Eq:Eric} on $[0,T)$. Due to continuity, ${\widetilde{q}}$ is
unique on
the compact interval $[0,T]$ as well.
\end{pf}

%
\begin{rem}\label{rem wugalter}
Condition (b) is equivalent to $\int_{|\xi| \ge1} e^{\langle
{y},{\xi}\rangle}
K(x,d\xi) < \infty$ for all $x \in D$, $y \in\RR^d$, that is, to the
jump measure having exponential moments of all orders. In this special
case analogues of Theorem~\ref{thmm:main_results} have been shown in
\citet{Spreij2010} and \citet{Cheridito2011}. This condition is
restrictive, as it is typically not satisfied in applications; cf.
Section~\ref{Sec:Examples}.
%
\end{rem}

We briefly discuss two important special cases, in which great
simplifications of the results occur. These cases have been treated
previously in the literature, but serve as a first ``sanity check'' of
the main results of this article.

%
\begin{example}[(Affine diffusion)]\label{ex filipovic}
Suppose that the affine process $X$ is a diffusion. In this case
$K(x,\cdot) = 0$ for all $x \in D$ and consequently $\cY= \RR^d$ and
the functions $F(y), R_1(y), \ldots, R_d(y)$ are quadratic
polynomials (hence locally Lipschitz continuous everywhere). In this
case any solution of the extended Riccati system is unique, and there
is no need to introduce the concept of minimal solutions; see
Proposition~\ref{Prop:minimal_solutions}(a) above.
Thus Theorem~\ref{thmm:main_results} holds true even with ``minimal
solution'' replaced by ``solution.'' For the case of affine diffusions
on canonical state spaces, the analogue of
Theorem~\ref{thmm:main_results} has been shown in
\citeauthor{Filipovic2009} [(\citeyear{Filipovic2009}), Theorem~3.3].
\end{example}

%
\begin{example}[(L\'evy process)]\label{ex levy}
Suppose that $X$ is a L\'evy process. Then $X$ is an affine process
with $R(y) = 0$ and with $F(y)$ equal to the L\'evy exponent of $X$.
Consequently $\cY$ is simply the effective domain of the L\'evy
exponent. The extended Riccati system has unique global solutions
for each $y \in\cY$, which are given by $p(t,y) = t F(y)$ and
$q(t,y) = y$ for $t \ge0$. It follows from
Theorem~\ref{thmm:main_results} that $\mathbb{E}^{x}[e^{\langle
{y},{X_t}\rangle}]$ is
finite if and only if $y \in\cY$, and in case of finiteness we have
$\mathbb{E}^{x}[e^{\langle{y},{X_t}\rangle}] = \exp(tF(y) + \langle
{y},{x}\rangle)$. In
particular, finiteness of exponential moments is a time-independent
property; that is, for given $y \in\RR^d$ the exponential moment
$\mathbb{E}^{x}[e^{\langle{y},{X_t}\rangle}]$ is either finite for
all $t > 0$ or for
no $t > 0$. Of course, all these results are well known in the case
of L\'evy processes and can be found, for example, in
\citeauthor{Sato1999} [(\citeyear{Sato1999}), Theorem~25.17].
\end{example}

\subsection{Complex moments of affine processes}
In this subsection we give an analogue of
Theorem~\ref{thmm:main_results} for complex exponential moments of
$X$. The first step is to analytically extend the functions $F$ and
$R$. We introduce the following notation: For a set $A \subset
\RR^d$ write
\[
S(A):= \bigl\{u \in\CC^d\dvtx \Re u \in A\bigr\}
\]
for the complex ``strip'' generated by $A$.

%
\begin{prop}\label{Prop:analytic_extension}
Let $X$ be an affine process, and suppose that $\cY^\circ\neq
\varnothing$. Then, for every $x \in D$, the function $\cR_x$ defined
in \eqref{Eq:Rx_def} has an analytic extension to $S(\cY^\circ)$
which we also denote by $\cR_x$. Moreover it holds that
\[
\cR_x(u) = F(u) + \bigl\langle{R(u)},{x}\bigr\rangle,\qquad  x \in D, u \in
S\bigl(\cY ^\circ\bigr),
\]
where $F,R$ are the analytic extensions of the functions defined in
\eqref{Eq:FR_def} to $S(\cY^\circ)$.
\end{prop}

\begin{pf}
Follows from standard results on L\'evy--Khintchine-type functions;
see, for example, \citet{Sato1999}, Theorem~25.17.
\end{pf}

%
\begin{defn}[(Complex Riccati system)]\label{Def:complex_riccati}
Let $X$ be an affine process such that $\cY^\circ\neq\varnothing$,
and let $F,R$ be defined as in Proposition~\ref{Prop:analytic_extension}.
Let $T \ge0, y \in S(\cY^\circ)$, and let
\[
\phi\dvtx  t \mapsto \phi(t,y),\qquad \psi\dvtx  t \mapsto\psi(t,y)
\]
be $C^1$-functions
mapping $[0,T]$ to $\CC$ [resp., $S(\cY^\circ)$] that satisfy
\begin{subequations}\label{Eq:Cric}
%
%
\begin{eqnarray}
\frac{\partial}{\partial t} \phi(t,y) &=& F\bigl(\psi(t,y)\bigr),\qquad \phi(0,y) =
0,\label{Eq:Cric1}
\\
\frac{\partial}{\partial t} \psi(t,y) &=& R\bigl(\psi(t,y)\bigr),\qquad \psi(0,y) = y
\label{Eq:Cric2}
\end{eqnarray}
\end{subequations}
for all $t \in[0,T]$. Then we call $(\phi,\psi)$ a solution (up to
time $T$ and with starting point $u$) of the complex Riccati system
associated to $X$.
\end{defn}

%
\begin{rem}\label{rem comparison}
Let us compare the complex Riccati system to the extended Riccati
system \eqref{Eq:Eric1}--\eqref{Eq:Eric2}. We observe that if $u \in
S(\cY^\circ)$ is real valued, that is, has $\Re u = y$ and $\Im u = 0$,
then any solution $(\phi,\psi)$ up to time $T$ of the complex
Riccati system is also a solution of the extended Riccati system; that
is, setting $p(t,y) = \phi(t,u)$ and $q(t,y) = \psi(t,u)$ for all
$t \in[0,T]$ defines a solution $(p,q)$ of the extended Riccati
system. The reverse is not necessarily true. Furthermore we point
out that for a given starting value $u$ any solution $(\phi,\psi)$
of the complex Riccati system is automatically the unique solution.
This is in contrast to the extended Riccati system, where solutions
starting at the boundary may be nonunique. This difference is just
a consequence of the fact that solutions of the complex Riccati
system are restricted to stay in the open domain $S(\cY^\circ)$, on
which $F$ and $R$ are locally Lipschitz.
\end{rem}

%
\begin{ass}\label{Ass:complex}
Let $X$ be an affine process with state space $D$ and assume that either:
\begin{longlist}[(i)]
\item[(i)]$D = \mathbb{R}_{\geq0}^m \times\RR^n$, or\vspace*{1pt}
\item[(ii)]
$D = S_d^+$ and there exists some
$x\in
S_d^{++}$ such that
$a(x)$ either vanishes, or it is nondegenerate.
\end{longlist}
\end{ass}

%
\begin{rem}
Note that in the notation of \eqref{Eq:characteristics_affine} $a(x)$
is given as a symmetric $\frac{d(d+1)}{2}\times\frac{d(d+1)}{2} $
matrix. Of course we can also interpret it as quadratic form on
$S_d^+$, which is more natural and, in particular, a coordinate free notion.
A simple characterization of (ii) in terms of the
admissible parameter set is given in Remark~\ref{rem assumption matrix}.
\end{rem}

The analogue of Theorem~\ref{thmm:main_results} for complex moments
reads as follows.

%
\begin{thmm}[(Complex moments of affine processes)]\label{thmm:main_complex}
Let $X$ be an affine process that satisfies Assumption~\ref
{Ass:complex}. Let $T \ge0$, $u \in S(\cY^\circ)$ and suppose that the
extended Riccati system \eqref{Eq:Eric1}--\eqref{Eq:Eric2}
has a solution $(p,q)$ with initial value $\Re u$ up to time $T$
such that $q(t,\Re u) \in\cY^\circ$ for all $t \in[0,T]$. Then
also the complex Riccati system \eqref{Eq:Cric} has a solution
$(\phi,\psi)$ with initial value $u$ up to time $T$,
$\mathbb{E}^{x}[|e^{\langle{u},{X_t}\rangle}|] < \infty$ and
%
%
\begin{equation}
\label{eq: complex transform formula} \mathbb{E}^{x}\bigl[e^{\langle{u},{X_t}\rangle}\bigr] = \exp\bigl(
\phi(t,u) + \bigl\langle{\psi (t,u)},{x}\bigr\rangle\bigr)
\end{equation}
for all $x \in D, t \in[0,T]$.
\end{thmm}

\section{Applications in mathematical finance}\label{sec:applications}
This section presents applications of our main results, Theorems
\ref{thmm:main_results} and~\ref{thmm:main_complex}, to
mathematical finance in the spirit of Duffie, Filipovi{\'c} and
Schachermayer (\citeyear{Duffie2003}), Section~13.
We consider the following
generic setup: A traded asset $S$ is modeled by the exponential of an
affine factor process $X$ with
state space $D$, that is, $S = e^{\langle{\theta},{X}\rangle}$ for
some $\theta
\in\RR^d$. Moreover, bond prices are given through an affine short
rate model of the form
\[
r_t = L(X_t)=l+\langle\lambda, X_t\rangle,
\]
where $l\in\mathbb R$ and $\lambda\in\mathbb R^d$. This setup
includes, in particular, affine term structure models of interest
rates [\citet{cir85}, \citet{duffiekan96,daisingleton00}, etc.], affine
stochastic volatility models [\citet{Heston1993}, \citet{Bates2000},
\citet{Barndorff2001}, etc.] and combinations with possible
correlation of
short rate and asset prices. Also credit risk can be included, when
$r_t$ is interpreted as a superposition of a risk-free short rate
and an affine default intensity process; cf. \citet{lando1998}.
Moreover, we can cover a setup with multiple possibly dependent assets
simply by setting $S^i = \exp{\langle{\theta_i},{X}\rangle}$ for
different $\theta
_i \in\RR^d$. For most applications the measures $(\PP^x)_{x \in D}$
should be considered risk-neutral measures, although there are few
cases where also the behavior under the physical measure is of
relevance. Many problems of interest can be reduced to determining the
$\mathcal
F_t$-conditional expectations
%
%
\begin{equation}
\label{Eq:Qt_pricing} Q_{T-t} g(x)=\mathbb E^x
\bigl[e^{-\int_t^T L(X_s)\,ds}g(X_T)\mid\mathcal F_t\bigr],
\end{equation}
for some measurable function $g\dvtx D \to\RR$. In
particular:
\begin{itemize}
\item$g \equiv1$ corresponds to bond pricing;
\item$g(x) = e^{\langle{\theta},{x}\rangle}$ corresponds to
checking for the
martingale property of the discounted asset price;
\item$g(x) = e^{\langle{y\theta},{x}\rangle}$, $y \in\RR$
corresponds to
calculating expectations of the type $\mathbb{E}^{x}[S_t^y]$ which are relevant
for evaluation of power utility and determining the time of ``moment
explosions.''
\item$g(x) = e^{\langle{u},{x}\rangle}$, $u \in\CC^d$ corresponds
to Fourier
methods for the pricing of European contingent claims.
\end{itemize}
For a more detailed account of the literature on affine processes in
financial mathematics, we refer to Duffie, Filipovi{\'c} and
Schachermayer (\citeyear{Duffie2003}), Section~13;
for an easy-to-read introduction to discounting and pricing
techniques (using the Fourier--Laplace transform), we refer to
\citet{Filipovic2009}, Section~4. Let us also remark that already
Duffie, Filipovi{\'c} and
Schachermayer (\citeyear{Duffie2003}), Section~11, gives sufficient conditions on an
affine process such that the pricing operator $Q_{T-t}$ is well defined,
but the results only apply to the state space $\mathbb{R}_{\geq0}^m
\times
\RR^n$ and conditions are less general than the ones we obtain.

To deal with the discounting term in \eqref{Eq:Qt_pricing} we use
the extension-of-state-space approach outlined in
Duffie, Filipovi{\'c} and
Schachermayer [(\citeyear{Duffie2003}), Section~11.2].
We define the extended state
space 
$\widetilde D:=D\times\mathbb R$. Let
$(a,\alpha,b,\beta,\break m(d\xi),\mu(d\xi))$ be the parameters of $X$ in
the sense of Theorem~\ref{Thm:semimartingale}. Following
Duffie, Filipovi{\'c} and
Schachermayer [(\citeyear{Duffie2003}), Section~11.2], we have that $Z:=(X,Y)$ where
$Y_t:=y+\int_0^t L(X_s)\,ds$ is an affine process on $\widetilde D$
with parameters $(a',\alpha',b',\beta',m'(d\xi),\mu'(d\xi))$ given
by
\[
a'=\pmatrix{ a&0
\vspace*{2pt}\cr
0&0 },\qquad \alpha'_i=\pmatrix{\alpha_i&0
\vspace*{2pt}\cr
0&0},\qquad b'=\pmatrix{ b
\vspace*{2pt}\cr
l}
\]
and
\[
\beta'_i=\pmatrix{
\beta_i
\vspace*{2pt}\cr
\lambda }, \qquad i=1,\ldots,d, \beta_{d+1}'=0,
\]
and finally
\[
m'(d\xi)=m(d\xi)\times\delta_0\bigl(d
\xi'\bigr), \qquad \mu_i'(d\xi)=
\mu_i(d\xi )\times \delta_0\bigl(d\xi'
\bigr),
\]
where $\delta_0(d\xi')$ denotes the unit mass at $0$. Let $F(u)$ and
$R(u)$ be the functions associated with $X$ through Proposition~\ref
{Prop:FR}, and let $q\in\mathbb C$. Then we can
introduce the new functions
\[
F'(u,q):=F(u)+lq, \qquad R'(u,q)=R(u)+\lambda q,
\]
which are related to the functions $(F_Z,R_Z)$ of the extended
process $Z$ in the way that $R_Z=(R'(u,q),0)$, while
$F_Z(u,q)=F'(u,q)$. We consider now solutions $\phi(t,u,q)$ and
$\psi(t,u,q)$ of the system
\begin{subequations}\label{Eq:Ric_lambda}
%
%
\begin{eqnarray}
\label{eqphi 1} \partial_t\phi(t,u,q)&=&F'\bigl(
\psi(t,u,q),q\bigr), \qquad \phi(0,u,q)=0,
\\
\label
{eqpsi 1} \partial_t\psi(t,u,q)&=&R'\bigl(
\psi(t,u,q),q\bigr),\qquad  \psi(0,u,q)=u.
\end{eqnarray}
\end{subequations}
Note that $\psi$ still is $d$-dimensional.
These solutions are related to the (not necessarily unique)
solutions $\phi_Z,\psi_Z$ of the corresponding $(d+2)$-dimensional
system associated with $F_Z,R_Z$ as follows:
$\phi_Z(t,(u,q))=\phi(t,u,q)$ and $\psi_Z(t,\break(u,q))=(\psi(t,u,q),q)$.

\subsection{Bond pricing in affine term structure models}

The following result is an immediate consequence of Theorem~\ref
{thmm:main_results}. As such it generalizes
Duffie, Filipovi{\'c} and
Schachermayer [(\citeyear{Duffie2003}), Proposition~11.2], as well as
\citeauthor{Filipovic2009} [(\citeyear{Filipovic2009}), Theorem~4.1].

%
\begin{thmm}\label{discounting thmm}
Let $\tau>0$. The following are equivalent:

\begin{longlist}[(1)]
\item[(1)] 
$\mathbb E^x[e^{-\int_0^\tau L(s)\,
ds}]<\infty$,
for some $x\in D^\circ$.
\item[(2)] 
For $q=-1$, there exists a solution
$(\widetilde\phi,\widetilde\psi)$ on $[0,\tau]$ to the generalized
Riccati differential equations \eqref{eqphi 1}--\eqref{eqpsi 1}
with initial data $u=0$.
\end{longlist}
In any of the above cases, let us define $A(t):=-\phi(t,(0,-1))$,
$B(t):= -\psi(t,\break(0,-1))$ from the unique minimal solution
$(\phi,\psi)$ of equations \eqref{eqphi 1}--\eqref{eqpsi
1}\mbox{.\hskip.2pt\footnote{It follows from Theorem~\ref{thmm:main_results} that if
some solution exists on a nonempty interval $[0,T]$, so does the
unique minimal solution.}} Then the price $P(t,T)$ of a zero-coupon
bond is given, for all $0\leq t\leq T\leq\tau$, and all $x\in D$,
by
%
%
\begin{equation}
P(t,T):=\mathbb E^x\bigl[e^{-\int_t ^T L(s)\,ds}\mid\mathcal
F_t\bigr]=e^{-A(T-t)-\langle B(T-t),X_t\rangle}.
\end{equation}
\end{thmm}

\subsection{Martingale conditions}
Conditions for the exponentials of affine processes to be martingales
have been obtained, for example, in \citet{mayerhofer}. The following
result extends known criteria and follows again from Theorem~\ref
{thmm:main_results}.

%
\begin{thmm}\label{martingale thmm}
Let ${\widetilde{S}} = e^{-\int_0^t L(X_t)} e^{\langle{\theta
},{X_t}\rangle}$ be the
discounted asset price. Then the
following holds:
\begin{longlist}[(1)]
\item[(1)] Suppose that $\theta\in\cY^\circ$, $F(\theta) = l$ and
$R(\theta) = \lambda$. Then $({\widetilde{S}}_t)_{t \ge0}$ is a true
martingale under any $\PP^x, x \in D$.
\item[(2)] Let $x \in D^\circ$. The process $({\widetilde{S}}_t)_{t
\ge0}$ is a true
$\PP^x$-martingale if and only if $\theta\in\cY$, $F(\theta) = l$,
$R(\theta) = \lambda$ and $\phi(t, \theta, -1)
= 0$ and $\psi(t, \theta, -1) = \theta$ are the unique minimal
solutions of
the Riccati equations \eqref{eqphi 1}--\eqref{eqpsi 1}.
\end{longlist}
\end{thmm}

Using \eqref{Eq:Qt_pricing} it is clear that ${\widetilde{S}}$ is a
$\PP
^x$-martingale if and only if $Q_t g(x) = g(x)$ for all $t \in\mathbb
{R}_{\geq0}$
and with $g(x) = e^{\langle{\theta},{x}\rangle}$. Applying
Theorem~\ref
{thmm:main_results} to the extended process $Z$ the above result
follows immediately.

\subsection{Moment explosions}
Here we set $L = 0$ for simplicity. It is well understood that the
existence of moments $\mathbb{E}[S_t^y]$ with $y \in\RR$ is intimately
connected to the shape of the implied volatility surface derived from
the prices of options on the underlying $S$; cf. \citet{Lee2004},
\citet
{K2008a}. Of particular interest is the time of moment explosion, that
is, the quantity
\[
T_+(y) = \sup\bigl\{t \ge0\dvtx \mathbb{E}\bigl[S_t^y\bigr]
< \infty\bigr\}.
\]
Applying again Theorem~\ref{thmm:main_results} we obtain the following:

%
\begin{prop}
Let $S = \exp{\langle{\theta},{X}\rangle}$ with $\theta\in\cY$,
and let $y \in
\RR$.
\begin{longlist}[(1)]
\item[(1)] If $y\theta\in\cY^\circ$, then $T_+(y)$ is the maximal
lifetime of the solution $(p,q)$ of the extended Riccati system.
\item[(2)] If $y\theta\in\cY$, then $T_+(y)$ is the maximal lifetime
of the unique minimal solution $(p,q)$ of the extended Riccati system.
If $y\theta\notin\cY$, then $T_+(y) = 0$.
\end{longlist}
\end{prop}

Related applications include the approximation of more complicated
payoff functions by ``power payoffs'' [see \citet{Cheridito2011}] and
portfolio optimization
involving power utility; see \citet{PhDMuhleKarbe} and the references
quoted therein.

\subsection{Option pricing}
In general, European option payoffs are nonlinear functions that do
not fall under the setup of the previous subsction. Numerically
expensive Monte Carlo simulations may be avoided by the method of
Fourier pricing, if the characteristic function (or Fourier--Laplace
transform) is given in closed form; cf. \citet{carrmadan99}. This is
the case for affine
processes, and the key for applying Fourier pricing is our Theorem~\ref
{thmm:main_complex} on complex exponential moments. We provide here
an extension of Theorem~10.5
from the book of \citeauthor{FilipovicBook} [(\citeyear{FilipovicBook}),
Chapter~10], which has been
written in the context of affine diffusions, where certain
simplifications occur (most importantly $\cY^\circ= \RR^d$). For
general affine processes with jumps we have to impose some stronger
assumptions and obtain the following result. To allow for multi-asset
options, we consider a generic payoff $g\dvtx D \to\RR$ depending on all
components of the underlying factor process $X$. In typical
applications $g$ will be of the more specific form $g(x) = h(e^{\langle
{\theta},{x}\rangle})$ with $h\dvtx \mathbb{R}_{\geq0}\to\RR$ which
can be accomodated in the
theorem below by setting $q = 1$; see also
\citeauthor{FilipovicBook} [(\citeyear{FilipovicBook}), Theorem~10.6].

%
\begin{thmm}\label{thmm:fourier_pricing}Let $X$ be an affine process
satisfying Assumption~\ref{Ass:complex}.
Assume there exists a $d\times
q$ matrix $K$ such that the payoff function $g$ satisfies
%
%
\begin{equation}
g(x)=\int_{\mathbb R^d} e^{\langle v+iK\lambda, x\rangle}\widetilde g(\lambda)\,d
\lambda
\end{equation}
for some integrable function $\widetilde g\dvtx \mathbb R^q\rightarrow
\mathbb C$, $q\leq d$ and with $v \in\cY^\circ$. Suppose that~\eqref
{eqphi 1}--\eqref{eqpsi 1} has solutions on $
[0,\tau] $ for initial data $u=0$ and $u=v$, which stay
in $\cY^\circ$ for all $t\leq\tau$. Then
we have
\[
\mathbb E^x\bigl[e^{-\int_t^T L(s)\,ds}g(X_T)\bigr]=\int
_{\mathbb R^q}e^{\phi
(T-t,v+iK\lambda)+\langle\psi(T-t,v+iK\lambda),X(t)\rangle
}\widetilde g(\lambda)\,d\lambda,
\]
where $(\phi,\psi)$ are the unique solutions of \eqref{eqphi
1}--\eqref
{eqpsi 1} with complex initial data $v+iK\lambda$.
\end{thmm}

%

\subsection{Remarks on jump behavior and examples}\label{Sec:Examples}
As discussed in the\break \hyperref[sec1]{Introduction}, a main contribution
of this article
is that the results apply to conservative affine processes with
completely general jump measures. The condition that the jump measures
possess exponential moments of all orders that is imposed in \citet
{Spreij2010} and \citet{Cheridito2011} is typically not fulfilled in
financial modeling. Considering, for example, the jump measures of the
models discussed in \citet{Cont2004}, Chapter~4, the condition is
satisfied only for the Merton model, but not for the Kou, variance
gamma, normal inverse Gaussian, tempered stable and generalized
hyperbolic models.

For some affine processes with jumps, the solution of the Riccati
equations is known explicitly. In this case, even when not all
exponential moments of the jump measures exist, ad-hoc arguments based
on analyzing the singularities of the characteristic function can be
used to find sufficient conditions for the validity of an affine
transform formula; see \citet{Nicolato2003} for an example of this
approach. While this ad-hoc approach does not give a satisfactory
answer on the connection between exponential moments and solutions to
the Riccati equations in general, it can be sufficient for
applications. However, as the following examples illustrate, several
models proposed in the literature on financial mathematics are based on
affine processes, for which the Riccati equations do not allow for
explicit solutions. In these cases previous results do not apply and
also the ad-hoc approaches fail. Hence, Theorems~\ref
{thmm:main_results} and \ref{thmm:main_complex} are essential for the
applications outlined in the previous sections and cannot replaced by
simpler arguments or existing results.

%
\begin{example}
\citet{Wu2011} models the S\&P 500 index as
\[
S_t=S_0 \exp(L_{\int_0^t v_u \,du}),\qquad  t\in[0,T],
\]
where $L$ is a L\'evy process of unit variance at unit time, and a
time-change is induced by a general $\mathbb{R}_{\geq0}$-valued
affine process
$v_t$-independent of $L$---with functional characteristics $F, R$;
see Definition~\ref{def: admin}.\hskip.2pt\footnote{\citet{Wu2011} also specifies
a separate drift term, which we absorb into the drift of the L\'evy
process~$L$.} We assume a riskless rate of return~$r$ and denote the
log-returns process by $Y_{t}:=\log(S_{t}/S_0)=L_{\int_0^t v_u \,du}$.
Writing $g$ for the characteristic exponent of $L$, we have
\[
\mathbb E\bigl[e^{iuY_t}\mid v_0=v\bigr]=e^{iu \theta t}
\mathbb E\bigl[e^{g(iu)\int
_0^t v_u \,du}\bigr]=e^{iu \theta t+\phi(t, g(iu))+v \psi(t, g(iu))},
\]
where $(\phi, \psi)$ satisfy
%
%
\begin{equation}
\label{ricci and wu} \frac{\partial}{\partial t} \phi=F(\psi),\qquad \frac{\partial
}{\partial t} \psi=R(\psi) +
\zeta, \qquad \psi(0)=\phi(0)=0
\end{equation}
with $\zeta=g(iu)$. Under the risk-neutral measure, $e^{-rt}S_t$ must
be a martingale on $[0,T]$, whence
\[
\mathbb E\bigl[e^{L_{\int_0^t v_u \,du}}\bigr]<\infty,
\]
for each $t\in[0,T]$. Theorem~\ref{thmm:main_results} implies that
the Riccati equations \eqref{ricci and wu} with $\zeta=g(1)$ allow a
minimal solution $(p,q)$ on $[0,T]$. Furthermore, if $q(t)$
lies in $\mathcal Y^\circ$ for each $t\in[0, T]$, an application of
Theorem~\ref{thmm:main_complex} extends the validity of the affine
property \eqref{ap} to complex moments $u\in(1+\varepsilon)+i
\mathbb
R$, where $\varepsilon>0$. Hence the way is paved for pricing
contingent claims
on $S$ by using, for example, the Fourier pricing technique.
\end{example}

%
\begin{example}
\citet{schneideretal} propose a model for pricing credit default swaps
(CDS), where the hazard rate
is a linear functional of an affine process $(\eta,\gamma)$, given
under the risk-neutral measure by
\begin{eqnarray*}
d\eta_t&=&\kappa_\eta(\gamma_t-
\eta_t)\,dt+\sigma_\eta\sqrt{\eta _t}
\,dW_{\eta,t} + dZ^1_t,
\\
d\gamma_t&=&\kappa_\gamma(\zeta_\gamma-
\gamma_t)\,dt+\sigma _\gamma\sqrt {\gamma_t}
\,dW_{\gamma,t} + dZ^2_t.
\end{eqnarray*}
The two components are correlated via the instantaneous drift and by
simultaneous jumps of the compound Poisson process $Z_t$. The jump-size
distributions of the two components are assumed to be independent and
exponentially distributed. Also here, a closed-form expression
for the characteristic function of $(\eta, \gamma)$ is not available.
However, due to the exponential distribution of jump sizes, the domain
$\mathcal Y^\circ$ takes a particular, simple form
\[
\mathcal Y^\circ=(-\infty, \mu_\eta)\times(-\infty,
\mu_\gamma),
\]
where $\mu_\eta,\mu_\gamma$ are the expected jump sizes of $\eta,
\gamma
$, respectively.
In this case, one first produces a numerical solution of the extended
Riccati system on $[0,T]$. By construction,
this solution will lie in $\mathcal Y^\circ$. Combining this
approximate solution with a global error bound we can find a $T' \le T$
such that also the exact solution must exist and stay in $\mathcal
Y^\circ$ on $[0,T']$. Theorem~\ref{thmm:main_results} then yields the
existence of the associated
real exponential moments. Having this solution, one can proceed
to solve the ODE with complex initial data. Existence of these
solutions and the validity of the corresponding affine transform formula
is guaranteed by Theorem~\ref{thmm:main_complex}, and Theorem~\ref
{thmm:fourier_pricing} can be used for Fourier pricing of contingent
claims in this case of credit default swaps.
An extension to state-dependent jump behavior is straightforward and
can be similarly dealt by using Theorem~\ref{thmm:main_complex}.
\end{example}

We give a final example of an asset model for optimal portfolio choice
with affine factors
which exhibit a nontrivial correlation structure:

%
\begin{example}
\citet{leippoldwu} propose an affine model $(Y,X)$, where
$Y_{i,t+h}-Y_{i,t}=\log(S_{i,t+h}/S_{i,t})$ ($i=1,\ldots,d$) are
log-returns for assets $S_i$ ($i=1,\ldots,d$),
and $X$ is a general $d\times d$ positive semidefinite affine
jump-diffusion. They specify $Y,X$ as a solution to the SDE
\begin{eqnarray*}
dY_t&=& \bigl[ r \mathbf{1}+X_t\eta-\tfrac{1}{2}
\diag(X_t) \bigr] \,dt+ \sqrt{X_t}\,dZ_t,
\\
dX_t&=&\bigl(\Omega\Omega^\top+M X_t+X_t
M^\top\bigr) \,dt+\sqrt{X_t}\,dB_t
Q+Q^\top dB_t^\top\sqrt{X_t}+dJ_t.
\end{eqnarray*}
Here $\diag(X_t)=(X_{t,11},\ldots,X_{t,dd})^\top$, $\mathbf
{1}=(1,\ldots,1)^\top$, both $Z$ and $B$ are $d\times d$ standard Brownian motions,
with a certain correlation structure defined by a correlation parameter
$\rho\in\mathbb R^d$; see \citet{leippoldwu} for details. Moreover, $J$
is a pure jump-process independent of $(B, W)$, whose jump intensity is
an affine function of $X_t$. The parameters are given by $\eta\in
\mathbb R^d$ and $M, \Omega, Q$ are $d\times d$ matrices satisfying the
constraint
$\Omega\Omega^\top-(d-1)Q^\top Q\in S_d^+$, which guarantees a weak
solution $(Y,X)$ to the above SDE. In this model, closed-form solutions
for the associated Riccati equations exist only in the absence of jumps
in $X$ (i.e., $J=0$).
\citet{leippoldwu} consider an investor with CRRA utility of terminal
wealth $w_T$, trading in each of these asset $S_i$ and with riskless
investment oportunity at constant rate $r>0$. It turns out that his/her
value function is given by
\[
V(t, w_t,X_t)=\frac{w_t^{1-\gamma}}{1-\gamma}\exp\bigl(\tr
\bigl(A(T-t)X_t\bigr)+B(T-t)\bigr),
\]
while the vector of optimal portfolio weights for the risky assets
equals $\pi^*(t)=(\eta+2 A(T-t) Q^\top\rho)/\gamma$.
The functions $A,B$ satisfy a matrix-valued Riccati equation, and
\citet
{leippoldwu} make the salient assumption
that (a) the value function is well defined at the optimal trading
policy $\pi^*(t)$, which amounts to assuming the existence of
exponential moments of the process $(X,Y)$ for a certain initial value;
(b) the associated extended Riccati equations are (uniquely) solvable
and give the value of these exponential moments. Theorem~\ref
{thmm:main_results} in this paper now gives sufficient and necessary
conditions that allows us to check the validity of these assumptions.
\end{example}

\section{Proofs for real moments of affine processes}\label{sec: real proofs}

\subsection{Decomposability and dependency on the starting value}\label{sec4.1}
Definition~\ref{Def:affine_process} of an affine process
immediately implies a decomposability property of the laws $\PP^x$
on the path space; see also Duffie, Filipovi{\'c} and
Schachermayer
[(\citeyear{Duffie2003}), Theorem~2.15]. As in
Duffie, Filipovi{\'c} and
Schachermayer [(\citeyear{Duffie2003}), Definition~2.14],
we write $\PP\star\PP'$ for the
image of $\PP\times\PP'$ under the measurable mapping $(\omega,
\omega') \mapsto\omega+ \omega'\dvtx (\Omega\times\Omega, \cF\times
\cF) \to(\Omega, \cF)$.

%
\begin{prop}\label{Prop:decompose}
Let $X$ be an affine process with state space $D$. Its probability
laws $\PP^x$ satisfy the following decomposability property: Suppose
that $x, \xi$ and $x + \xi$ are in $D$. Then
%
%
\begin{equation}
\label{Eq:decompose} \PP^x \star\PP^\xi= \PP^0
\star\PP^{x + \xi}.
\end{equation}
\end{prop}

\begin{pf}
Write $\mathbf{u} = (u^1, \ldots, u^N)$ for an ordered set of points
$u^k \in\cU$. Choosing some finite sequence $0 \le t_1 \le\cdots
\le t_N$ in $\mathbb{R}_{\geq0}$, define
\[
f(x,\mathbf{u}) = \mathbb{E}^{x}\Biggl[\exp\Biggl(\sum
_{k=1}^N \bigl\langle {X_{t_k}},{u^k}
\bigr\rangle\Biggr)\Biggr]\qquad \bigl(x \in D, \bigl(u^1, \ldots,
u^N\bigr) \in\cU^N\bigr),
\]
that is, $f(x,\mathbf{u})$ is the joint characteristic function of
$(X_{t_1}, \ldots, X_{t_n})$ under $\PP^x$. Applying the affine
property \eqref{Eq:affine_property} recursively, we obtain
%
%
\begin{equation}
\label{Eq:f_function} f(x,\mathbf{u}) = \exp\bigl(p(\mathbf{u}) + \bigl\langle{x},{q(
\mathbf{u})}\bigr\rangle\bigr),
\end{equation}
where $p(\mathbf{u}) = p_1$ and $q(\mathbf{u}) = q_1$, with
%
%
\begin{eqnarray}
p_{k-1} &= &\phi\bigl(t_k - t_{k-1},q_k
+ u^k\bigr) + p_k,\qquad p_N = 0,
\\
q_{k-1} &= &\psi\bigl(t_k - t_{k-1},q_k
+ u^k\bigr),\qquad q_N = 0.
\end{eqnarray}
From \eqref{Eq:f_function} we derive that
\[
f(x, \mathbf{u}) f(\xi, \mathbf{u}) = f(0,\mathbf{u}) f(x + \xi,\mathbf{u})
\]
for all $\mathbf{u} = (u^1, \ldots, u^N) \in\cU^N$. Since the
distribution of a stochastic process is determined by its
finite-dimensional marginal distributions, this equality is
equivalent to \eqref{Eq:decompose}.
\end{pf}

In the following, we set
\[
g(t,y,x) = \mathbb{E}^{x}\bigl[e^{\langle{y},{X_t}\rangle}\bigr] = \int
_D{e^{\langle{y},{\xi }\rangle}p_t(x,d\xi)},
\]
for all $(t,y) \in\mathbb{R}_{\geq0}\times\RR^d$ and $x \in D$.
Note that
$g(t,y,x)$ is always strictly positive, but might take the value
$+\infty$. By approximating $g(t,y,x)$ monotonically from below by
bounded functions and using the Chapman--Kolmogorov equation, we derive that
%
%
\begin{equation}
\label{Eq:Markov_unbounded} g(t+s,y,x) = \int_D g(t,y,\xi)
p_s(x,d\xi)
\end{equation}
holds for all $t,s \in\mathbb{R}_{\geq0}$, $y \in\RR^d$ and $x \in
D$, where
$+\infty$ is allowed on both sides and in the integrand. The
following lemma concerns the role of the starting value $X_0 = x$ of
the affine process with regards to finiteness of exponential
moments.

%
\begin{lem}\label{Lem:x_reduction}Let $X$ be an affine process on $D$,
and let $(T,y) \in\mathbb{R}_{\geq0}\times\RR^d$. Then the
following holds:
\begin{longlist}[(a)]
\item[(a)]$\mathbb{E}^{0}[e^{\langle{y},{X_T}\rangle}] = \infty$
implies $\mathbb{E}^{x}[e^{\langle{y},{X_T}\rangle}] =
\infty$ for all $x \in D^\circ$;
\item[(b)]$\mathbb{E}^{x}[e^{\langle{y},{X_T}\rangle}] < \infty$ for
some $x \in D^\circ$
implies $\mathbb{E}^{x}[e^{\langle{y},{X_T}\rangle}] < \infty$ for
all $x \in
D$;
\item[(c)]$\mathbb{E}^{x}[e^{\langle{y},{X_T}\rangle}] < \infty$ for
all $x \in D$ implies
$\mathbb{E}^{x}[e^{\langle{y},{X_t}\rangle}] < \infty$ for all $t
\in[0,T], x \in
D$.
\end{longlist}
\end{lem}

\begin{pf}As before we set $g(t,y,x) = \mathbb{E}^{x}[e^{\langle
{y},{X_t}\rangle}]$,
which takes values in the extended
positive half-line $(0,\infty]$. Using the decomposability property of
$X$ (cf.
Proposition~\ref{Prop:decompose}), we have
\[
g(t,y,x)g(t,y,\xi)=g(t,y,0)g(t,y,x+\xi)
\]
for all $x,\xi\in D$ for which $x+\xi\in D$. Let $x_*$ be an arbitrary
point in $D$.
Setting $x = \xi= x_*/2$ it follows that
%
%
\begin{equation}
\label{eq:functional_eq} g \biggl(t,y,\frac{x_*}{2}\biggr)^2=g(t,y,0)g(t,y,x_*).
\end{equation}
We conclude that $g(t,y,0) = \infty$ implies $g(t,y,\frac{x_*}{2}) =
\infty$; hence (a) is verified for the point $x =
x_*/2$. We introduce the affine process
$X_{2,t}:=X_t-x_*/2$ with state-space $D_2:=D - x_*/2$. Clearly $0\in
D_2$ and $g_2(t,y,z):=\mathbb{E}[\exp(\langle{y},{X_{2,t}}\rangle
)|X_{2,0}=z] = e^{-\langle{y},{x_*}\rangle/2} g(t,y,z+x_*/2)$.
Since $g_2(t,y,0)=\infty$ and $x_*/2\in D_2$, we may apply the same
argument as above to $g_2$ instead of $g$ and obtain that
$g_2(t,y,x_*/4)=\infty$. But this means $g(t,y,3x_*/4)=\infty$, and by
iterating this procedure, we obtain that for each $k\geq1$
$g(t,y,x_*(1-2^{-k}))=\infty$. Since $x_*$ was an arbitrary point in
the convex set $D$, (a) follows.
We prove (b) by contraposition: Assume that
$g(t,y,x) = \infty$ for some $x \in D$, and introduce the affine process
$Z_t:=X_t-x$ with state-space $D_x:=D -x$. Clearly $0\in D_x$ and
$\mathbb{E}[e^{\langle{y},{Z_t}\rangle}|Z_0=0]=e^{\langle
{y},{x}\rangle} g(t,y,x) = \infty$,
by assumption. Applying (a) to the process $Z$
yields that
\[
g(t,y,\xi) = e^{-\langle{y},{\xi}\rangle} \mathbb{E}\bigl[e^{\langle
{y},{Z_t}\rangle}|Z_0=
\xi-x\bigr] = \infty
\]
for all $\xi\in D$, which completes the proof of (b).


To show (c) pick an arbitrary $x \in
D^\circ$ and $\varepsilon> 0$. Since $X$ has c\`adl\`ag paths, we can
find $\delta> 0$ such that $\PP^{x}(\Vert X_t - x\Vert < \varepsilon)
\ge
\tfrac{1}{2}$ for all $t
\le\delta$. With $p_t(x,d\xi)$ denoting the transition kernel of
$X$, we can rewrite this as $p_t(x,B_\varepsilon(x)) \ge\tfrac{1}{2}$ for
all $t \le\delta$.
We show assertion (c) for $t \in[T -\delta,
T]$; the general case follows then by iteration. By \eqref
{Eq:Markov_unbounded},
\[
g(T,y,x) = \int_D g(t,y,\xi) p_{T-t}(x,d\xi)
\]
holds for all $(t,y) \in[0,T] \times\RR^d$. By assumption, the
left-hand side is finite, and we want to show that also $g(t,y,\xi)$ is
finite for all $\xi\in D$ and $t \in[T -\delta,
T]$. Assume for a contradiction that
$g(t,y,\xi^*) = \infty$ for some $t \in[T-\delta,T]$ and $\xi^*
\in D$.
Then by Lemma~\ref{Lem:x_reduction}(b)
$g(t,y,\xi) = \infty$ for all $\xi\in D^\circ$. But
$p_{T-t}(x,D^\circ) \ge p_{T-t}(x,B_\varepsilon(x)) \ge\tfrac{1}{2}$,
and we conclude that $g(T,y,x) = \infty$, which is a contradiction.
\end{pf}

\subsection{From moments to Riccati equations}
In this section we prove
Theorem~\ref{thmm:main_results}(a), except for the
minimality property of the Riccati solution.

%
\begin{lem}\label{Lem:proto_moment} Let $X$ be an affine process on
$D$, and let $T \ge0$. Suppose that for some
$x \in D^\circ$ and $y \in\RR^d$, it holds that
$\mathbb{E}^{x}[e^{\langle{y},{X_T}\rangle}] < \infty$. Then $y \in
\cY$, and the
following holds:
\begin{longlist}[(a)]
\item[(a)]
There exist functions $t \mapsto
p(t,y) \in\RR$ and $t \mapsto
q(t,y) \in\RR^d$ such that \eqref{eq: transform formula} holds for
all $x\in D, t\in[0,T]$.
\item[(b)]
$\mathbb{E}^{x}[e^{\langle
{q(T-t,y)},{X_t}\rangle}] <
\infty$ for all $t \in[0,T]$ and
%
%
\begin{equation}
\label{Eq:proto_moment_cond_real}\qquad \mathbb{E}^{x}\bigl[e^{\langle{y},{X_T}\rangle}|\cF_t
\bigr] = \exp\bigl(p(T-t,y) + \bigl\langle{q(T-t,y)},{X_s}\bigr
\rangle \bigr) \qquad\mbox{for all $x \in D$.}
\end{equation}
\item[(c)]
The functions $p(t,y), q(t,y)$ satisfy
the semi-flow equations
\begin{subequations}\label{Eq:semiflow}
%
%
\begin{eqnarray}
p(T,y) &=& p(T-t,y) + p\bigl(t,q(T-t,y)\bigr),\qquad p(0,y) = 0,
\\
q(T,y) &=& q\bigl(t,q(T-t,y)\bigr), \qquad q(0,y) = y,
\end{eqnarray}
\end{subequations}
for all $t \in[0,T]$.
\end{longlist}
\end{lem}

\begin{pf}
From Lemma~\ref{Lem:x_reduction}(c) it follows
that $\mathbb{E}^{x}[e^{\langle{y},{X_t}\rangle}] < \infty$ for all
$(t,x) \in[0,T]
\times D$. Fix $t \in[0,T]$, and write $g(t,y,x) =
\mathbb{E}^{x}[e^{\langle{y},{X_t}\rangle}]$. Then by
Proposition~\ref{Prop:decompose}
$g(t,y,x)$ satisfies the functional equation \eqref{Eq:decompose}.
Since\vspace*{1pt} $g(t,y,0) > 0$ there exists $p(t,y) \in\RR$ such that
$g(t,y,0) = e^{p(t,y)}$. Set $h(t,y,x) = e^{-p(t,y)}g(t,\break y,x)$. Then
$h(t,y,x)$ is finite for all $x \in D$ and satisfies Cauchy's
functional equation
\[
h(t,y,x)h(t,y,\xi) = h(t,y,x + \xi),\qquad x, \xi, x + \xi\in D.
\]
We conclude that there exists $q(t,y) \in\RR^d$ such that $h(t,y,x)
= e^{\langle{q(t,y)},{x}\rangle}$ for all $x \in D$, and we have shown
\eqref{eq: transform formula}.

To show equation~\eqref{Eq:proto_moment_cond_real} note that by the
Markov property of $X$,
\[
\mathbb{E}^{x}\bigl[\mathbf{1}_{\{|X_t| \le n\}} e^{\langle
{y},{X_t}\rangle}|
\cF_s\bigr] = \int_{\{\xi \in D\dvtx |\xi| \le n\}} e^{\langle{y},{\xi}\rangle}
p_{t-s}(X_s,d\xi)
\]
holds for all $n \in\NN, x \in D, 0 \le s \le t$. Using dominated
convergence we may take the limit $n \to\infty$ and obtain equation
\eqref{Eq:proto_moment_cond_real} from \eqref{eq: transform
formula}. Taking (unconditional) expectations in \eqref
{Eq:proto_moment_cond_real} yields
\begin{eqnarray*}
&&\exp\bigl(p(T,y) +\bigl\langle{x},{q(T,y)}\bigr\rangle\bigr)
\\
&&\qquad = \exp\bigl(p(T-t,y) + p\bigl(t,q(T-t,x)\bigr) + \bigl\langle{x},{q
\bigl(t,q(T-t,y)\bigr)}\bigr\rangle\bigr),
\end{eqnarray*}
for all $t \ge0$ and $x \in D$. Since $D$ contains $0$ and linearly
spans $\RR^d$, the semi-flow equations \eqref{Eq:semiflow} follow.
\end{pf}

Note that if $t \mapsto p(t,y)$ and $t \mapsto q(t,y)$ are
differentiable with derivatives $F(y)$ and $R(y)$ at zero, then it
follows by differentiating the semi-flow
equations~\eqref{Eq:semiflow} that $(p,q)$ is a solution of the
extended Riccati system \eqref{Eq:Eric}. The main difficulty thus is
showing the differentiability of $p$ and $q$. This is very similar
to the \emph{regularity problem} for affine processes, where the
same question is asked regarding the functions $\phi(t,u)$ and
$\psi(t,u)$ in Definition~\ref{Def:affine_process}. Several
solutions of the regularity problem have been given; see, for example,
\citet{KST2009} and \citet{KST2011}.
Here we adapt the approach of \citet{Cuchiero2011} to our setting.

We enlarge the probability space $(\Omega, \cF, \FF, \PP)$ such that
it supports $d + 1$ independent copies of the affine process $X$,
which we denote by $X^0, \ldots, X^d$. Without loss of generality it
can be assumed that $X = X^0$. In what follows we will use the
convention that upper indices correspond to the different instances of
the process~$X$, while lower indices correspond to the coordinate
projections of a single process. For a vector $x^i \in D$ denote by
$\PP^{x^i}$ the probability $\PP$ conditional on
$\{X^i(0) = x^i\}$; that is, the process $X^i$ starts at the point
$x^i$ with $\PP^{x^i}$-probability $1$. Similarly for an ordered set
$\mathbf{x} = (x^0, \ldots, x^d)$ of points in $D$, we denote by
$\PP^{\mathbf{x}}$ the probability $\PP$ conditional on $\{(X^0(0) =
x^0) \wedge\cdots\wedge(X^d(0) = x^d)\}$; that is, the processes $X^0,
\ldots, X^d$ start at the points $x^0, \ldots, x^d$ respectively with
$\PP^{\mathbf{x}}$-probability~$1$.

%
\begin{lem}\label{Lem:invertible}
Let $X^0, \ldots, X^d$ be $d+1$ independent copies of the affine
process~$X$. Furthermore let $\mathbf{x} = (x^0, \ldots, x^d)$ be
$d+1$ affinely independent points in $D$. Define the matrix-valued
random function
\[
\Xi(t;\mathbf{x},\omega) = \pmatrix{ 1 &
X_1^0(t,\omega) & \cdots& X_d^0(t,
\omega)
\vspace*{2pt}\cr
\vdots& \vdots& \ddots& \vdots
\vspace*{2pt}\cr
1 & X_1^d(t,\omega) & \cdots& X_d^d(t,
\omega) }.
\]
Then there exists $\delta> 0$ such that
\[
\PP^{\mathbf{x}} \bigl(\det\Xi(t;\mathbf{x}) \neq0 \mbox{ for all } 0 \le t \le
\delta\bigr) > \tfrac{1}{2};
\]
that is, $t \mapsto\Xi(t;\mathbf{x})$ stays regular on $[0,\delta]$
with $\PP^{\mathbf{x}}$-probability at least $\frac{1}{2}$.
\end{lem}

\begin{pf}
Since $(x^0, \ldots,x^d)$ are affinely independent, and $X^i(0) =
x^i$ for all $i \in\{0, \ldots, d\}$ with
$\PP^{\mathbf{x}}$-probability one, the matrix $\Xi(0;\mathbf{x})$
is regular $\PP^\mathbf{x}$-almost surely. Define $a_t = \inf_{s \in
[0,t]} | \det\Xi(s;\mathbf{x})|$. Since the processes
$X^i$ are right-continuous, also $\det\Xi(s;\mathbf{x})$ is, and
hence even $a_t$. By dominated convergence also $b_t =
\PP^{\mathbf{x}}(a_t > 0)$ is right-continuous and has the starting
value $b_0 = 1$. We conclude that there exists some $\delta> 0$
such that $b_\delta> \frac{1}{2}$, which completes the proof.
\end{pf}

The following proposition settles
Theorem~\ref{thmm:main_results}(a) apart from the
minimality property of $(p,q)$ as solutions of the extended Riccati
system. The key ideas in the subsequent proof come from
\citet{Cuchiero2011}, proofs of Lemma~1.5.3, Theorem~1.5.4.

%
\begin{prop}\label{part of main theorem part a}
Let $X$ be an affine process on $D$, and let $T \ge0$. Let $y \in
\RR^d$, and suppose that $\mathbb{E}^{x}[e^{\langle{y},{X_T}\rangle
}] < \infty$ for
some $x \in D^\circ$. Then $y \in\cY$ and there exist a solution
$(p,q)$ up to time $T$ of the extended Riccati system
\eqref{Eq:Eric}, such that \eqref{eq: transform formula} holds for
all $x \in D$, $t \in[0,T]$.
\end{prop}

\begin{pf}
Recall that we are working on an extended probability space that
supports $d + 1$ independent copies $(X^0, \ldots, X^d)$ of $X$. Let
$\mathbf{x} = (x^0, \ldots, x^d)$ be $d + 1$ affinely independent
points in $D$. By Theorem~\ref{Thm:semimartingale} each $X^i$ is a
$\PP^{\mathbf{x}}$-semi-martingale with canonical semimartingale
representation
%
%
\begin{equation}
\label{Eq:semi_rep} X_t^i = x^i + \int
_0^t{b\bigl(X^i_{s-}
\bigr)\,ds} + N_t^i + \int_0^t
\int_{\RR^d} \bigl(\xi- h(\xi)\bigr) J^i(
\omega;ds,d\xi),
\end{equation}
where $N_t^i$ is a local martingale and $J^i(\omega;dt,d\xi)$ is the
Poisson random measure associated to the jumps of $X^i$ with
predictable compensator $K(X^i_{t-},d\xi)\,dx$.

By Lemma~\ref{Lem:proto_moment} we know that
%
%
\begin{equation}
\label{Eq:M_martingale} \mathbb{E}^{\mathbf{x}}\bigl[e^{\langle{y},{X^i_T}\rangle}|\cF_t
\bigr] = \exp\bigl(p(T-t,y) + \bigl\langle{q(T-t,y)},{X_t^i}
\bigr\rangle\bigr)
\end{equation}
for each $i \in\{0,\ldots,d\}, t \in[0,T]$. Let us denote $M_t^i
= \mathbb{E}^{\mathbf{x}}[e^{\langle{y},{X^i_T}\rangle}|\cF_t]$.
Clearly, each $t
\mapsto M_t^i$ is a $\PP^{\mathbf{x}}$-martingale for $t \le T$ and
for each $i \in\{0, \ldots, d\}$. Taking logarithms and arranging
the equations in matrix form, we get
%
%
\begin{equation}
\pmatrix{ \log M_t^0(\omega)
\vspace*{2pt}\cr
\vdots
\vspace*{2pt}\cr
\log M_t^d(\omega) } =
\pmatrix{ 1 & X_1^0(t,
\omega) & \cdots& X_d^0(t,\omega)
\vspace*{2pt}\cr
\vdots& \vdots& \ddots& \vdots
\vspace*{2pt}\cr
1 & X_1^d(t,\omega) & \cdots& X_d^d(t,
\omega)} \cdot \pmatrix{
p(T-t,y)
\vspace*{2pt}\cr
q_1(T-t,y)
\vspace*{2pt}\cr
\vdots
\vspace*{2pt}\cr
q_d(T-t,y)},\hspace*{-30pt}
\end{equation}
and recognize on the right-hand side the matrix $\Xi(t;\mathbf
{x},\omega
)$ from Lemma~\ref{Lem:invertible}. The latter allows us to conclude
that there exists a set $A \subset\Omega$ with $\PP^\mathbf{x}(A) >
\frac{1}{2}$ and some $\delta> 0$
such that $\Xi(t;\mathbf{x},\omega)$ is invertible for all $t \in
[0,\delta]$ and $\omega\in A$. Hence for $T' = T \wedge\delta$ we obtain
%
%
\begin{eqnarray}
\label{matrixinv} &&\pmatrix{ 1 & X_1^0(t,
\omega) & \cdots& X_d^0(t,\omega)
\vspace*{2pt}\cr
\vdots& \vdots& \ddots& \vdots
\vspace*{2pt}\cr
1 & X_1^d(t,\omega) & \cdots& X_d^d(t,
\omega) }^{-1} \cdot\pmatrix{\log M_t^0(\omega)
\vspace*{2pt}\cr
\vdots
\vspace*{2pt}\cr
\log M_t^d(\omega)}
\nonumber
\\[-8pt]
\\[-8pt]
\nonumber
&&\qquad=
\pmatrix{ p\bigl(T'-t,y\bigr)
\vspace*{2pt}\cr
q_1\bigl(T'-t,y\bigr)
\vspace*{2pt}\cr
\vdots
\vspace*{2pt}\cr
q_d\bigl(T'-t,y\bigr)},
\end{eqnarray}
for all $t \in[0,T']$. All processes occurring on the left-hand
side of equation~\eqref{matrixinv} are semimartingales, hence also
the right-hand side consists row-by-row of semimartingales for all
$t \in[0,T']$. Since they are deterministic, the functions $t
\mapsto p(t,y)$ and $t \mapsto q(t,y)$ are of finite variation on
$[0,T']$. This implies in particular that they are almost everywhere
differentiable and can be written as
\begin{subequations}\label{Eq:proto_riccati}
%
%
\begin{eqnarray}
p\bigl(T'-t,y\bigr) - p\bigl(T',y\bigr) &=& - \int
_0^t{\,dp\bigl(T'-s,y\bigr)},
\\
q\bigl(T'-t,y\bigr) - q\bigl(T',y\bigr) &= &- \int
_0^t{\,dq\bigl(T'-s,y\bigr)}.
\end{eqnarray}
\end{subequations}
Applying It\^{o}'s formula to the martingales $M_t^{i,y}$, we obtain
\begin{eqnarray*}
M_t^{i,y}& =& M_0^{i,y} + \int
_0^t M_{s-}^{i,y} \bigl(-dp
\bigl(T'-s,y\bigr) + \bigl\langle {-dq\bigl(T'-s,y
\bigr)},{X^i_{s-}}\bigr\rangle \bigr)
\\
&&{}+ \int_0^t M_{s-}^{i,y}
\biggl\{ \bigl\langle{q\bigl(T'-s,y\bigr)},{b\bigl(X^i_{s-}
\bigr)}\bigr\rangle \\
&&\hspace*{56pt}{}+ \frac
{1}{2}\bigl\langle{q\bigl(T'-s,y
\bigr)},{a\bigl(X^i_{s-}\bigr) q\bigl(T'-s,y
\bigr)}\bigr\rangle
\\
&&\hspace*{56pt}{}+ \int_{D} \bigl(e^{\langle{q(T'-s,y)},{\xi}\rangle} - 1 - \bigl\langle {q
\bigl(T'-s,y\bigr)},{h(\xi )}\bigr\rangle\bigr)\\
&&\hspace*{194pt}{}\times K
\bigl(X^i_{s-},d\xi\bigr) \biggr\} \,ds
\\
&&{}+ \int_0^t M_{s-}^{i,y}
\bigl\langle{q\bigl(T'-s,y\bigr)},{dN^i_s}
\bigr\rangle
\\
&&{}+ \int_0^t \int_{D}
M_{s-}^{i,y} \bigl(e^{\langle{q(T'-s,y)},{\xi
}\rangle} - 1 - \bigl\langle{q
\bigl(T'-s,y\bigr)},{h(\xi)}\bigr\rangle\bigr)\\
&&\hspace*{40pt}{}\times \bigl(J(\omega,ds,d
\xi) - K\bigl(X^i_{s-},d\xi\bigr)\,ds\bigr),
\end{eqnarray*}
for all $i \in\{0,\ldots, d\}$. On the right-hand side, the last
two terms are local martingales and the other terms are of finite
variation. Hence the finite variation terms have to sum up to $0$.
Rewriting in terms of the functions $F(y)$ and $R(y)$ this means
that
\begin{eqnarray*}
&&-dp\bigl(T'-t,y\bigr) + \bigl\langle{-dq\bigl(T'-t,y
\bigr)},{X^i_{s-}}\bigr\rangle\\
&&\qquad = F\bigl(q
\bigl(T'-t,y\bigr)\bigr) \, dt + \bigl\langle{X^i_{s-}},{R
\bigl(q\bigl(T'-t,y\bigr)\bigr)}\bigr\rangle \,dt
\end{eqnarray*}
holds for almost all $t \in[0,T']$ $\PP^{\mathbf{x}}$-a.s.
Inserting into \eqref{Eq:proto_riccati} and using the regularity of
the matrix $\Xi(t,\mathbf{x},\omega)$ on $[0,T']$, this yields
%
%
\begin{eqnarray}
\label{Eq:proto_riccati_2} p\bigl(T'-t,y\bigr) - p\bigl(T',y
\bigr) &=& - \int_0^t{F\bigl(q
\bigl(T'-s,y\bigr)\bigr)\,ds},
\\
q\bigl(T'-t,y\bigr) - q\bigl(T',y\bigr) &=& - \int
_0^t{R\bigl(q\bigl(T'-s,y\bigr)
\bigr)\,ds}.
\end{eqnarray}
Applying the fundamental theorem of calculus, we have shown that
$(p,q)$ is a solution to the extended Riccati system \eqref{Eq:Eric} up
to $T' = T \wedge\delta$, where $\delta$ was given by Lemma~\ref
{Lem:invertible}.
To show the general case we conclude with an induction argument.
Suppose that $(p,q)$ are solutions of the extended Riccati system up to
$T_k = T \wedge(k \delta)$. We show that they
can be extended to solutions up to $T_{k+1} = T \wedge((k+1) \delta)$.
Set $\Delta_k = T_{k+1} - T_k$; clearly $\Delta_k \le\delta$. By
Lemma~\ref{Lem:x_reduction},
$\mathbb{E}^{x}[e^{\langle{y},{X_T}\rangle}] < \infty$ implies that
$\mathbb{E}^{x}[e^{\langle{y},{X_{T_{k+1}}}\rangle}] < \infty$, and by
Lemma~\ref{Lem:proto_moment} we have that
$\mathbb{E}^{x}[e^{\langle{q(y,T_k)},{X_{\Delta_k}}\rangle}] <
\infty$. Set $y' =
q(y,T_k)$. Then, proceeding exactly as in the proof above, we obtain
\begin{subequations}
%
%
\begin{eqnarray}
\frac{\partial}{\partial t}p\bigl(t,y'\bigr)& = &F\bigl(q\bigl(t,y'
\bigr)\bigr), \qquad p\bigl(0,y'\bigr) = 0,
\\
\frac{\partial}{\partial t}q\bigl(t,y'\bigr)& =& R\bigl(q\bigl(t,y'
\bigr)\bigr),\qquad  q\bigl(0,y'\bigr) = y'
\end{eqnarray}
\end{subequations}
for all $t \in[0,\Delta_k]$. Using the flow property, this is
equivalent to
\begin{subequations}\label{Eq:Riccati_end}
%
%
\begin{eqnarray}
\frac{\partial}{\partial t}p(t,y) &= &F\bigl(q(t,y)\bigr),\qquad p(0,y) = 0,
\\
\frac{\partial}{\partial t}q(t,y) &=& R\bigl(q(t,y)\bigr), \qquad q(0,y) = y
\end{eqnarray}
\end{subequations}
for all $t \in[T_k,T_{k+1}]$. By the induction hypothesis
\eqref{Eq:Riccati_end} already holds for all $t \in[0,T_k]$, and we
have shown that $(p,q)$ is a solution of the extended Riccati system
up to $T_{k+1} = T \wedge\delta(k+1)$. As this holds true for all
$k \in\NN$, the proof is complete.
\end{pf}

\subsection{From Riccati equations to moments}

Using the result from above, the step from the extended Riccati
system to the existence of moments is simple:

%
\begin{prop}\label{prop thb}
Let $X$ be an affine process taking values in $D$. Let $y \in\cY$,
and suppose that the extended Riccati system \eqref{Eq:Eric} has a
solution $({\widetilde{p}}, {\widetilde{q}})$ that starts at $y$ and
exists up to $T
\ge0$. Then $\mathbb{E}^{x}[e^{\langle{y},{X_T}\rangle}] < \infty$
and \eqref{eq:
transform formula} holds for all $x \in D$, $t \in[0,T]$, where
$(p,q)$ is also a solution up to $T$ to \eqref{Eq:Eric}.
\end{prop}

\begin{pf}
Using the solution $({\widetilde{p}},{\widetilde{q}})$ of the
extended Riccati system
\eqref{Eq:Eric}, define for $t \in[0,T]$,
%
%
\begin{equation}
\label{Eq:M_supermartingale} {\widetilde{M}}^{y}_t = \exp\bigl({
\widetilde{p}}(T-t,y) + \bigl\langle {{\widetilde{q}}(T-t,y)},{X_t}
\bigr\rangle\bigr).
\end{equation}
Applying It\^{o}'s formula to ${\widetilde{M}}^{y}_t$ and using the semimartingale
representation \eqref{Eq:semi_rep}, we see that
\begin{eqnarray*}
{\widetilde{M}}^{y}_t& =& {\widetilde{M}}_0^y
+ \int_0^t {\widetilde {M}}_{s-}^y
\bigl(-{\widetilde{p}}(T-s,y) + \bigl\langle{- {\widetilde{q}}(T-s,y)},{X_{s-}}
\bigr\rangle \bigr) \,ds
\\
&&{}+ \int_0^t {\widetilde{M}}_{s-}^y
\biggl( \bigl\langle{{\widetilde {q}}(T-s,y)},{b(X_{s-})}\bigr\rangle
\\
&&\hspace*{55pt}{}+ \frac
{1}{2}\bigl\langle{{\widetilde{q}}(T-s,y)},{a(X_{s-}) {
\widetilde {q}}(T-s,y)}\bigr\rangle
\\
&&\hspace*{55pt}{}+ \int_{D} \bigl(e^{\langle{{\widetilde{q}}(T-s,y)},{\xi}\rangle} - 1 - \bigl\langle{{
\widetilde{q}}(T-s,y)},{h(\xi)}\bigr\rangle\bigr) K(X_{s-},d\xi)
\biggr) \,ds
\\
&&{}+ \int_0^t {\widetilde{M}}_{s-}^y
\bigl\langle{{\widetilde {q}}(T-s,y)},{dN_s}\bigr\rangle
\\
&&{}+ \int_0^t \int_{D} {
\widetilde{M}}_{s-}^y \bigl(e^{\langle{{\widetilde
{q}}(T-s,y)},{\xi}\rangle} - 1 - \bigl
\langle{{\widetilde{q}}(T-s,y)},{h(\xi)}\bigr\rangle\bigr)\\
&&\hspace*{42pt}{}\times \bigl(J(\omega,ds,d
\xi) - K(X_{s-},d\xi)\,ds\bigr).
\end{eqnarray*}
The $ds$-terms can be simplified to
\begin{eqnarray*}
&&-{\widetilde{p}}(T-s,y) + \bigl\langle{- {\widetilde {q}}(T-s,y)},{X_{s-}}
\bigr\rangle + F\bigl({\widetilde{q}}(T-s,y)\bigr) + \bigl\langle{R\bigl({
\widetilde{q}}(T-s,y)\bigr)},{X_{s-}}\bigr\rangle \\
&&\qquad= 0,
\end{eqnarray*}
and we conclude that $({\widetilde{M}}^{y}_t)_{t \in[0,T]}$ is a
local $\PP
^x$-martingale for all $x \in D$. It is also strictly positive, and
hence it is a
$\PP^x$-supermartingale. Therefore
\[
\mathbb{E}^{x}\bigl[e^{\langle{y},{X_T}\rangle}\bigr] = \mathbb
{E}^{x}\bigl[{\widetilde{M}}_T^{y}\bigr] \le{
\widetilde{M}}_0^{y} < \infty
\]
for all $x \in D$. The second part of the assertion, and in particular
the validity of equation \eqref{eq: transform formula} follows now by applying
Proposition~\ref{part of main theorem part a}.
\end{pf}

\subsection{Proof of Theorem~\texorpdfstring{\protect\ref{thmm:main_results}}{2.14}}
Looking at
Proposition~\ref{prop thb} and Proposition~\ref{part of main theorem
part a} we see that Theorem~\ref{thmm:main_results} is almost
proved. Only one issue in both parts of the theorem is not answered
yet, namely the minimality of $(p,q)$ in \eqref{eq: transform
formula} as minimal (hence unique, see Remark~\ref{rem112}) solution
of the extended Riccati system. We start with the following lemma:

%
\begin{lem}\label{finallem}
Let $(p,q)$ and $({\widetilde{p}},{\widetilde{q}})$ be given as in
Proposition~\ref{prop thb}.
Then for all $t \in[0,T]$ and $x \in D$,
\[
p(t,y) + \bigl\langle{q(t,y)},{x}\bigr\rangle \le{\widetilde{p}}(t,y) + \bigl
\langle{{\widetilde{q}}(t,y)},{x}\bigr\rangle.
\]
\end{lem}

\begin{pf}
Set $M_t^y = \exp(p(T-t,y) + \langle{q(T-t,y)},{x}\rangle)$, and
define ${\widetilde{M}}_t^{y}$ as in~\eqref{Eq:M_supermartingale}.
Then, for
each $x \in D$ the process $M^y$ is a $\PP^x$-martingale\vspace*{1pt} [see \eqref
{Eq:M_martingale} and below]; ${\widetilde{M}}^y$ is a $\PP
^x$-supermartingale,
and they satisfy $M_T^{y} = {\widetilde{M}}_T^{y}$.
Hence
\[
M_t^{y} = \mathbb{E}^{x}\bigl[M_T^{y}|
\cF_t\bigr] = \mathbb{E}^{x}\bigl[{\widetilde
{M}}_T^{y}|\cF_t\bigr] \le{
\widetilde{M}}_t^{y}
\]
for all $t \in[0,T]$. Taking logarithms the claimed inequality follows.
\end{pf}

\begin{pf*}{Proof of Theorem~\ref{thmm:main_results}}
Proof of (a): In view of Remark~\ref{rem112} we only
need to show that the solution $(p,q)$ of the Riccati system
established in Proposition~\ref{part of main theorem part a} is minimal.
Let $(\widetilde p,\widetilde q)$ be another solution on $[0,T']$ of
the extended Riccati system, $T'\leq T$. Then by Proposition~\ref{prop
thb} there exists $(p^*,q^*)$ such that~\eqref{eq: transform formula}
holds for all $y\in D$ and $t\in[0,T']$, as is the case for $(p,q)$.
By taking logarithms of the respective right-hand sides of \eqref{eq:
transform formula} and by applying Lemma~\ref{finallem}, we see that
\[
p+\langle q,y\rangle=p^*+\bigl\langle q^*,y\bigr\rangle\leq\widetilde p+\langle
\widetilde q,x\rangle,
\]
on $[0, T']$ and for all $y\in D$. Hence by Definition~\ref
{definminsol} $(p,q)$ is the minimal solution of the extended Riccati
system, and we are done with part (a).

The proof of (b) follows immediately from Lemma~\ref
{finallem}, Definition~\ref{definminsol} and Remark~\ref{rem112}.
\end{pf*}

\section{Proofs for complex moments of affine processes}\label{sec: complex}
In this section we show Theorem~\ref{thmm:main_complex} on the
existence of complex moments of affine processes, whose state space
satisfies Assumption~\ref{Ass:complex}. The key to the proof is to
relate the lifetime of solutions $(\phi,\psi)$ of the complex
Riccati system \eqref{Eq:Cric1}--\eqref{Eq:Cric2} and the solutions
$(p,q)$ of the extended Riccati system
\eqref{Eq:Eric1}--\eqref{Eq:Eric2}. Unlike in preceding parts of the
paper, we only solve for initial values in the interiors
$y\in\cY^\circ$ [resp., $u\in S(\cY^\circ)$]. Also, in this
section we
need more precise knowledge about the restrictions on the
parameters, which appear in the Riccati equations.

With $S_d^+$ we denote the $d\times d$ positive semidefinite matrices.
Let $T_+(y)$ [resp., $T_+(u)$] denote the maximal lifetime of $t\mapsto
(p(t,y),q(t,y))$ [resp., $t\mapsto(\phi(t,u),\psi(t,u))$].

%
\begin{prop}\label{c blow up after r}
Suppose that Assumption~\ref{Ass:complex} holds true, and let $u\in
S(\cY^\circ)$ and $y=\Re(u)$. Then $T_+(u)\geq T_+(y)$.
\end{prop}

We split the proof into the two cases covered by Assumption~\ref
{Ass:complex}, a state space $D$ of the form $\mathbb{R}_{\geq0}^m
\times\RR^n$
and a state space of the form $S_d^+$. Note that $\mathbb{R}_{\geq
0}^m \times\RR
^n$ and $S_d^+$ are convex cones. To apply certain results of \citet
{Volkmann1973} on multivariate ODE comparison, we introduce the
following property:

%
\begin{defn}\label{Def:quasi_monotone}
Let $K\subset\mathbb R^d$ be a proper closed, convex cone, and denote
by $\preceq$ the induced partial order. Let $U\subset\mathbb R^d$. A function
$f\dvtx U\rightarrow\mathbb R^d$ is called quasimonotone increasing (with
respect to $K$), if for all $y,z\in U$ for which $y\preceq z$
and $\langle y, x\rangle=\langle z, x\rangle$ for some $x\in K$ it
holds that $\langle f(y), x\rangle\leq\langle f(z), x\rangle$.
\end{defn}

\subsection{State space \texorpdfstring{$D=\mathbb{R}_{\geq0}^m\times\RR^n$}{D=\mathbb{R}_{>=0}^mtimes\mathbb{R}^n}}
In this section\vspace*{2pt} we consider the ``canonical state space'' $D =
\mathbb{R}_{\geq0}^m \times\RR^n$ from
Duffie, Filipovi{\'c} and
Schachermayer
(\citeyear{Duffie2003}). We use
the index
sets $I=\{1,2,\ldots,m\}$ and $J=\{m+1,\ldots,d\}$ corresponding to
the positive and to the real valued components of $D$ respectively.
Accordingly, $R_I$ denotes the function $(R_1,\ldots, R_m)$, and
similarly $R_J$ is constituted by the last $n$ coordinates of $R$.

First, we recall the definition of the admissible parameter set for
(conservative) affine processes on $\mathbb{R}_{\geq0}^m \times\RR
^n$ from Duffie, Filipovi{\'c} and
Schachermayer (\citeyear
{Duffie2003}):

%
\begin{defn}\label{def: admin}
A set of $\RR^d$-vectors $b,\beta^1,\ldots,\beta^d$, positive
semidefinite $d \times d$ matrices $a,\alpha^1,\ldots,\alpha^d$, L\'
evy~measures
$m,\mu^1,\ldots,\mu^d$ on $\RR^d$, is
called \emph{admissible for $D = \mathbb{R}_{\geq0}^m \times\RR
^n$} if and only if
\begin{eqnarray*}
a_{kl} &=& 0 \qquad\mbox{for all } k\in I \mbox{ or } l\in I,
\\
\alpha^j&=&0 \qquad\mbox{for all } j\in J,
\\
\alpha^i_{kl} &= &0\qquad \mbox{if } k\in I\setminus\{i\}
\mbox{ or } l\in I\setminus\{i\};\\
b &\in& D,
\\
\beta^i_k - \int\xi_k\mu^i(d\xi)
&\ge&0 \qquad\mbox{for all } i\in I, k\in I\setminus\{i\},
\\
\beta^i_k &=& 0 \qquad\mbox{for all } j\in J, k\in I;\\
\int_{|\xi|\le1} |\xi_I| m(d\xi) &<& \infty,
\\
\mu^j &=& 0\qquad \mbox{for all } j\in J,
\\
\int_{|\xi|\le1} {|\xi_{I\setminus\{i\}}|} \mu^i(d\xi)
&<& \infty.
\end{eqnarray*}
\end{defn}

%
\begin{rem}
The matrices $a,\alpha^1,\ldots,\alpha^d$ are frequently referred to as
\emph{diffusion matrices}, the vectors $b,\beta^1,\ldots,\beta^d$
as
\emph{drift vectors} and the L\'evy~measures $m,\mu^1,\ldots,\mu^d$ as
\emph{jump measures}.
\end{rem}

Let $R(y) = (R_1(y), \ldots, R_d(y) )$ be defined as in
Proposition~\ref{Prop:FR}. The admissibility conditions imply that each
$R_1(y), \ldots, R_d(y)$ is a convex lower semi-continuous function of
L\'evy--Khintchine-type. Denoting $\mu^0(d\xi):=m(d\xi)$ we
therefore have
%
%
\begin{equation}
\label{simplfied D canonical} \cY= \Biggl\{y \in\RR^d\dvtx \sum
_{i=0}^d \int_{|\xi| \ge1}
e^{\langle
{y},{\xi}\rangle} \mu^i(d\xi) < \infty\Biggr\},
\end{equation}
which is the intersection of the effective domains of $F, R_1, \ldots,
R_d$.

We start with the following crucial lemma:

%
\begin{lem}\label{Lem:R_properties_complex}
There exists a function $g$ which is finite, nonnegative and
convex on $\cY$ such that for all $u\in S(\cY^\circ)$ we have
%
%
\begin{equation}
\label{Item:Rnorm} \Re\bigl(\bigl\langle{\overline{u}_I},{R_I(u)}
\bigr\rangle\bigr) \le g(\Re u) \bigl(1 + |u_J|^2\bigr)
\bigl(1 + |u_I|^2\bigr).
\end{equation}
\end{lem}

\begin{pf}
It clearly sufficient to show
\[
\Re\bigl(\overline{u}_i R_i(u)\bigr) \le g_i
\bigl(\Re(u)\bigr) \bigl(1 + |u_J|^2\bigr) \bigl(1 +
|u_I|^2\bigr)
\]
individually for each $i \in I$ and with some nonnegative convex
$g_i(\cdot)$ that is finite on $\cY$. In addition we may split $R_i(u)$
into the drift part, the diffusion part, a small-jump part and a
large-jump part and show the inequality for each part separately.
The drift and the large jump-part are the easiest to deal with.
Using the Cauchy--Schwarz inequality we infer the existence of a
positive constant $C$ such that
%
%
\begin{eqnarray}
\label{Eq:estimate_drift} \Re\bigl(\overline{u}_i \bigl(\bigl\langle{
\beta^i},{u}\bigr\rangle\bigr)\bigr) &\le& |u_i|\bigl(\bigl|
\beta _I^i\bigr||u_I| + \bigl|\beta_J^i\bigr||u_I|
\bigr)
\nonumber
\\[-8pt]
\\[-8pt]
\nonumber
&\le& C\bigl(1 + |u_J|^2\bigr) \bigl(1 +
|u_I|^2\bigr)
\end{eqnarray}
and
%
%
\begin{eqnarray}
\label{Eq:estimate_large_jump} \Re\biggl(\overline{u}_i \int_{|\xi| > 1}
e^{\langle{\xi},{u}\rangle} \mu^i(d\xi ) \biggr) &\le&|u_i| \int
_{|\xi| > 1} e^{\langle{\xi},{\Re u}\rangle} \mu^i(d\xi)
\nonumber
\\[-8pt]
\\[-8pt]
\nonumber
&\le&{\widetilde{g}}_i(\Re u) \bigl(1 + |u_I|^2
\bigr)
\end{eqnarray}
for the large-jump part. Here ${\widetilde{g}}_i(z):= \int_{|\xi| > 1}
e^{\langle{\xi},{z}\rangle} \mu^i(d\xi)$ clearly is a nonnegative convex
function which is finite on $\cY$. To estimate the diffusion part we
have to take into account the admissibility conditions, which tell
us that $\alpha^i_{ij}$ is zero if $j \in I \setminus\{i\}$. Thus
we obtain
%
%
\begin{eqnarray}
\label{Eq:estimate_diffusion} \qquad\Re\bigl(\overline{u}_i \bigl\langle{u},{
\alpha^i u}\bigr\rangle \bigr) &=& \alpha^i_{ii}
|u_i|^2 \Re u_i + 2 \Re\bigl(|u_i|^2
\alpha_{iJ} u_J\bigr) + \Re\bigl(\overline{u}_i
u_J^\top\alpha^i_{JJ}
u_J\bigr)
\nonumber
\\[-8pt]
\\[-8pt]
\nonumber
&\le& C\bigl(1 + (\Re u_I)_+\bigr) \bigl(1 +
|u_J|^2\bigr) \bigl(1 + |u_I|^2
\bigr),
\end{eqnarray}
as desired. The hardest term to estimate is the small-jump part. We
follow the proof of Lemma~6.2 in Duffie, Filipovi{\'c} and
Schachermayer
(\citeyear{Duffie2003}). As a shorthand
notation we introduce $u_{I-} = u_{I \setminus\{i\}}$ and $u_{J+}
= u_{J \cup\{i\}}$. First we do a Taylor expansion of
the integrand $h(\xi) = e^{\langle{\xi},{u}\rangle} - 1 -
\langle{\xi_{J+}},{u_{J+}}\rangle$ with $|\xi| \le1$,
%
%
\begin{eqnarray}
h(\xi) &=& e^{\langle{\xi},{u}\rangle} - e^{\langle{\xi
_{J+}},{u_{J+}}\rangle} + e^{\xi_i
u_i}
\bigl(e^{\langle{\xi_J},{u_J}\rangle} - 1 - \langle{\xi _J},{u_J}
\rangle\bigr)
\nonumber\\
\nonumber
&&{} + \langle{\xi_J},{u_J}\rangle
\bigl(e^{\xi_i u_i} - 1\bigr) + e^{\xi
_i u_i} - 1 - \xi_i
u_i
\\
& =& e^{\langle{\xi_{J+}},{u_{J+}}\rangle} \biggl(\int_0^1{e^{t\langle{\xi _{I-}},{u_{I-}}\rangle}
\,dt}\biggr) \langle {u_I},{\xi_I}\rangle
\\
\nonumber
&&{} + e^{\xi_i u_i} \biggl(\int_0^1{(1
- t)e^{t \langle{\xi
_J},{u_J}\rangle}\,dt}\biggr) \sum_{j,k \in J}
\xi_j \xi_k u_j u_k
\\
\nonumber
&&{} + \biggl(\int_0^1{e^{t \xi_i u_i}
\,dt}\biggr) \xi_i u_i \sum_{j \in J}
\xi_j u_j + \biggl(\int_0^1{(1
- t)e^{t \xi_i
u_i}\,dt}\biggr) \xi_i^2
u_i^2.
\end{eqnarray}
Next we calculate
\[
\Re\bigl(\overline{u}_i h(\xi)\bigr) = K(u,\xi) +
|u_i|^2 \xi_i \int_0^1{(1-t)
\Re\bigl(u_i \xi_i e^{t u_i \xi_i}\bigr)\,dt}.
\]
Since $|\xi| \le1$, we get
%
%
\begin{eqnarray}
\label{Eq:estimateK} \bigl|K(u,\xi)\bigr| &\le& e^{(\Re u)_+} \bigl(|u_i|^2
+ |u_I||u_J|^2 + |u_I||u_J|
\bigr) \bigl(|\xi_{I-}| + |\xi_{J+}|^2\bigr)
\nonumber
\\[-8pt]
\\[-8pt]
\nonumber
&\le&\bigl(1 + e^{(\Re u)_+}\bigr) \bigl(1 + |u_J|^2
\bigr) \bigl(1 + |u_I|^2\bigr) \bigl(|
\xi_{I-}| + |\xi_{J+}|^2\bigr)
\end{eqnarray}
for the first term. For the second term we use
Lemma~\ref{Lem:complex_ineq} below and estimate
%
%
\begin{equation}
\label{Eq:estimateL} |u_i|^2 \xi_i \int
_0^1{(1-t)\Re\bigl(u_i
\xi_i e^{t u_i \xi_i} \bigr)\,dt} \le|u_i|^2
\xi_i \bigl(e^{\xi_i (\Re u_i)_+} - 1\bigr).
\end{equation}
Adding up \eqref{Eq:estimateK} and \eqref{Eq:estimateK}, and
integrating against the L\'evy measure $\mu^i$ we obtain
%
%
\begin{equation}
\label{Eq:estimate_small_jump} \Re\biggl(\overline{u}_i \int_{|\xi| \le1}
{h(\xi) \mu^i(d\xi)}\biggr) \le{\widehat{g}}_i(\Re u)
\bigl(1 + |u_J|^2\bigr) \bigl(1 + |u_I|^2
\bigr)
\end{equation}
with
\[
{\widehat{g}}_i(y) = e^{y_+} \int_{|\xi| \le1}
{\bigl(|\xi_{I-}| + |\xi _{J+}|^2\bigr)
\mu^i(d\xi)} + \int_{|\xi| \le1} \xi_i
\bigl(e^{\xi_i
y_i} - 1\bigr)\mu^i(d\xi),
\]
which is nonnegative, convex and finite for all $y \in\RR^d$.
Adding up \eqref{Eq:estimate_drift}--\eqref{Eq:estimate_diffusion} and \eqref{Eq:estimate_small_jump}
yields the desired estimate \eqref{Item:Rnorm}.
\end{pf}

%
\begin{lem}\label{Lem:complex_ineq}
For any $z \in\CC$,
%
%
\begin{equation}
\label{Eq:complex} \int_0^1{(1 - t)\Re\bigl(z
e^{tz}\bigr)}\,dt \le\bigl(e^{(\Re z)_+} - 1\bigr).
\end{equation}
\end{lem}

\begin{pf}
For $\Re z \le0$ the inequality was shown in
Duffie, Filipovi{\'c} and
Schachermayer (\citeyear{Duffie2003}). Denote
the left-hand side by $L(z)$. Writing $z = p + iq$ and evaluating
the integral, we have that
\begin{eqnarray*}
L(z) &=& \int_0^1(1-t)e^{pt}\bigl(p
\cos(qt) - q \sin(qt)\bigr)
\\
&=&\frac{1}{p^2 + q^2} \bigl\{ p \bigl(e^p \cos(q) - 1 - p
\bigr) + q\bigl(e^p \sin(q) - q\bigr) \bigr\}.
\end{eqnarray*}
The expression is symmetric in $q$ such that we may assume that $q
\ge0$. If in addition $p \le0$, then using $\cos(q) \ge1 - q^2/2$
and $\sin(q) \le1$ we may estimate
\begin{eqnarray*}
L(z) \le\frac{1}{p^2 + q^2}\bigl(p\bigl(e^p - 1 - p\bigr) -
q^2 + q^2 e^p (1 - p/2)\bigr).
\end{eqnarray*}
Since $e^p(1 - p/2) \le1$ and $(e^p - 1 - p) \ge0$, the right-hand
side is smaller than $0$ showing the lemma for $p \le0$. If $p \ge
0$, we may use that $\cos(q) \le1$, $\sin(q) \le q$ and $e^p - 1 \le
pe^p$ to estimate
\[
L(z) \le\frac{1}{p^2 + q^2} \bigl\{p^2\bigl(e^p - 1
\bigr) + q^2\bigl(e^p - 1\bigr)\bigr\} =
\bigl(e^p - 1\bigr),
\]
thus completing the proof.
\end{pf}

Recall Definition~\ref{Def:quasi_monotone} of quasimonotonicity with
respect to a convex cone $K$. Here, $K=\mathbb R^m_+$; in this
particular setting, quasimonotonicity of a function $f\dvtx U \subset K \to
\RR^m$ can be expressed in coordinates and is equivalent to
\[
y\preceq z,\mbox{ and } y_i=z_i \mbox{ for some } i
\in\{1,\ldots,m\} \Rightarrow f_i(y)=f_i(z).
\]

%
\begin{lem}\label{qmi}
Let $y_J\in\mathbb R^n$. For each $t\geq0$, $y_I\mapsto
R_I(y_I,\psi_J(t,y_J))$ is quasimonotone increasing (with respect to
the natural cone $\mathbb R_+^m$) on $\cY$.
\end{lem}

\begin{pf}
See, for instance, \citet{keller} or \citet{mayerhofer}.
\end{pf}

We further need the following special property of $\cY^\circ$.

%
\begin{lem}\label{order regular D canonical}
If $y\in\cY^\circ$, $z\in\mathbb R^d$ and $z_I\preceq y_I$,
$z_J=y_J$, then we also have $z\in\cY^\circ$.
\end{lem}

\begin{pf}
We choose $\varepsilon>0$ such that $B_{\varepsilon}(y)=\{w\in
\mathbb R^d\mid|y-w|<\varepsilon\}\subset\cY$. By~\eqref{simplfied D canonical} we have for $i=0,1,\ldots,d$,
%
%
\begin{equation}
\label{est: moments interior} \int_{|\xi| \ge1} e^{\langle{y+w},{\xi}\rangle}
\mu^i(d\xi) < \infty,\qquad |w|<\varepsilon.
\end{equation}
Note the L\'evy measures $\mu^i$ are clearly positive and supported on
$D$. Now for all $\xi\in D$ we have
\[
\langle z+w,\xi\rangle=z_I^\top\xi_I+z_J^\top
\xi_J+\langle w,\xi \rangle =z_I^\top
\xi_I+ y_J^\top\xi_J+\langle
w,\xi\rangle\leq\langle y+w,\xi \rangle
\]
because $\xi_I\in\mathbb R_+^m$ and $z_i\leq y_i$ for all $i\in I$,
by assumption. Hence, by the monotonicity of the exponential we see
that \eqref{est: moments interior} holds with $y$ replaced by $z$.
Hence, once again by \eqref{simplfied D canonical} we have
$B_{\varepsilon}(z)\subset\cY$, that is, $z\in\cY^\circ$.
\end{pf}

We are now prepared to prove Proposition~\ref{c blow up after r} under
Assumption~\ref{Ass:complex}(i).

\begin{pf*}{Proof of Proposition~\ref{c blow up after r} under
Assumption~\ref{Ass:complex}(\normalfont{i})}
By a straightforward check, for every $u\in\cY$,
\[
\Re\bigl(R_i(u)\bigr)\leq R_i\bigl(\Re(u)\bigr),
\]
and by Lemma~\ref{qmi} we can apply the ODE comparison result of
\citet
{Volkmann1973} to the first $m$ coordinates of $\psi$, which let us
conclude that
$\Re(\psi_I(t,u))\preceq\psi_I(t,\Re(u))$ for $t<T_+(u)\wedge
T_+(\Re
(u))$. In view of Lemma~\ref{order regular D canonical} we therefore
have $T_+(u)\geq T_+(\Re(u))$, unless $|\psi(t,u)|$
explodes before\break $|\psi(t,\Re(u))|$ does. We show in the following that
this cannot happen: By Lemma~\ref{Lem:R_properties_complex} we have
\begin{eqnarray*}
\frac{\partial}{\partial t} \bigl|\psi_I(t,u)\bigr|^2&=&2\Re\bigl\langle
\overline {\psi}_I(t,u), R_I\bigl(\psi(t,u)\bigr)\bigr
\rangle
\\
&\leq &g\bigl(\Re\psi(t,u)\bigr) \bigl(1 + \bigl|\psi_J(t,u)\bigr|^2
\bigr) \bigl(1 + \bigl|\psi_I(t,u)\bigr|^2\bigr)
\end{eqnarray*}
with a function $g$ which is finite on all of $\cY$. Since
$\psi_J(t,u)\equiv\psi_J(t,u_J)$ exists globally as solution of a
linear ordinary differential equation, we obtain by Gronwall's
inequality applied to $(1 + |\psi_I(t,u)|^2)$ that
\[
\bigl|\psi_I(t,u)\bigr|\leq|u_I|^2+
\bigl(1+|u_I|^2\bigr)\int_0^t
h(s)e^{\int_0^s
h(\xi)\,d\xi}\,ds,
\]
where $h(t):=g(\Re\psi(t,u)) (1 + |\psi_J(t,u)|^2)$. Hence we
have shown $T_+(u)\geq T_+(\Re(u))$.
\end{pf*}

\subsection{Matrix state spaces}\label{Sub:matrix_cone}
Let $S_d$ be the space of symmetric real $d\times d$ matrices,
endowed with the inner product $\langle{x},{y}\rangle = \tr(xy)$,
where $\tr$
denotes the trace operator. We further denote by $\mathbb
C^{m\times n}$ the space of complex $m\times n$ matrices. We make
the latter into a normed space by introducing a norm as
$\|a\|^2:=\tr(a \bar a^\top)$. Here $^\top$ denotes matrix
transposition, and $\bar a$ is the element-wise conjugate of the
matrix $a$.

We start with the following observation, which is a generalization
of \citet{MayerhoferFinVar}, Lemma B.1.

%
\begin{lem}\label{involvemelemma}
There exists a locally Lipschitz function $h\dvtx S_d\rightarrow\mathbb
R_+$ such that for all $a\in\mathbb C^{m\times n}$ and for any $b\in S(S_n)$,
we have
%
%
\begin{equation}
\label{firstmatrixlem} \Re\tr\bigl(-b\bar a a^\top\bigr)\leq h(\Re b)\cdot\|a
\|^2.
\end{equation}
\end{lem}

\begin{pf}
Recall that the projection $\pi\dvtx  S_d\rightarrow S_d^+$ is a well
defined, convex (hence locally Lipschitz continuous) map, which
satisfies $\pi(z) \succeq z$ for all $z \in S_d$.

Let us write $a=a_1+ia_2$ and $b=b_1+ib_2$ with $a_1,a_2\in\mathbb
R^{m\times n}$ and $b_1,b_2\in S_n$. Then we have
\begin{eqnarray*}
\Re\tr\bigl(-b\bar a ^\top a\bigr)&=&\Re\tr\bigl(-(b_1+ib_2)
\bigl(a_1^\top-ia_2^\top \bigr)
(a_1+ia_2)\bigr)
\\
&=&\tr\bigl(-b_1\bigl(a_1^\top a_1
\bigr)\bigr)+\tr\bigl(-b_1 \bigl(a_2^\top
a_2\bigr)\bigr)+0
\\
&\leq&\tr\bigl(\pi(-b_1) \bigl(a_1^\top
a_1\bigr)\bigr)+\tr\bigl(\pi(-b_1) \bigl(a_2^\top
a_2\bigr)\bigr)
\\
&\leq&\bigl\|\pi(-b_1)\bigr\| \bigl( \| a_1\|^2+\|
a_2\|^2\bigr)
\\
&=&\bigl\|\pi(-b_1)\bigr\| \|a\|^2.
\end{eqnarray*}
The last inequality holds in view of the Cauchy--Schwarz inequality. We
now see that inequality
\eqref{firstmatrixlem} holds by setting $h(x):=\|\pi(-x)\|$.
\end{pf}

Next we present the admissibility conditions for matrix-valued affine
processes that have been established in \citet{cfmt}. Note that in the
case $d = 1$ it holds that $S_d^+ = \mathbb{R}_{\geq0}$, that is, the
one-dimensional case is already covered by the previous section.
Therefore we may assume that $d \ge2$, which leads to several
simplifications of the parameter conditions. It has been shown in
\citet
{MayerhoferFinVar} that affine processes on $S_d^+$ ($d\geq2$) do not
exhibit jumps of infinite total variation. Compared with \citet{cfmt}
this makes the use of a truncation function
in the definition of $R$ obsolete and also simplifies the very
complicated (i.e., hard to check) necessary tradeoff between linear
jump coefficient and drift; cf.~\citet{cfmt}, 2.11. In the following,
$\preceq$ denotes the partial order on $S_d$ induced by the cone $S_d^+$.

\begin{defn}
An admissible parameter set $(\alpha, b,B, m(d\xi), \mu(d\xi))$
consists
of:
\begin{itemize}
\item a linear diffusion coefficient $\alpha\in S_d^+$,
\item a constant drift $b\in S_d^+$ satisfying
\[
b\succeq(d-1)\alpha,
\]
\item a constant jump term: a Borel measure m on $S_d^+\setminus\{0\}$
satisfying
\[
\int_{S_d^+\setminus\{0\}}\bigl(\|\xi\|\wedge1\bigr)m(d\xi)<\infty,
\]
\item a linear jump coefficient $\mu$ which is an $S_d^+$-valued,
sigma-finite measure
on $S_d^+\setminus\{0\}$ satisfying
\[
\int_{S_d^+\setminus\{0\} }\bigl(\|\xi\|\wedge1\bigr)\mu(d\xi)<\infty
\]
\item and finally, a linear drift $B$, which is a linear map from $S_d$
to $S_d$ and
``inward pointing'' at the boundary of $S_d^+$. That is,
\[
\tr\bigl(x B(u)\bigr)\geq0\qquad \mbox{for all } u,x \in S_d^+\mbox{
with }\tr(ux )=0.
\]
\end{itemize}
\end{defn}

\begin{rem}\label{rem assumption matrix}
Using the notation $a(x)$ from \eqref{Eq:characteristics_affine} we
have $a(x)(u)=2\tr(x u\alpha u)$, and the following are equivalent:
\begin{longlist}[(1)]
\item[(1)] 
condition \ref{Ass:complex}(ii);
\item[(2)] 
$\alpha=0$ or $\alpha$ is invertible;
\item[(3)] 
either $a(x)$ vanishes for all $x\in S_d^+$,
or it is nondegenerate for any $x\in S_d^{+}\setminus\{0\}$.
\end{longlist}
The only nontrivial direction to prove is (1) $\Rightarrow
$ (2). Assume, for a contradiction, that $\alpha\neq0$,
but $\alpha$ is degenerate. Then there exists
$u\in S_d^+\setminus\{0\}$ such that $u\alpha=\alpha u=0$. But then
$a(x)(u)=\tr( x u\alpha u)=0$, for any $x$.
\end{rem}

Note that \citet{cfmt} uses the Laplace transform to define the affine
property, which introduces several changes of signs compared with our
definition. To comply with the notation of \citet{cfmt}, we introduce
\[
\wh F(y) = - F(-y), \qquad \wh R(y) = -R(-y),
\]
which can now be written as
\begin{eqnarray*}
\wh F(y)&=&\tr(b y)-\int_{S_d^+\setminus\{0\}}\bigl(e^{-\tr(y\xi
)}-1
\bigr)m(d\xi),
\\
\wh R(y)&=&-2y\alpha y+B^\top(y)-\int_{S_d^+\setminus\{0\}}
\bigl(e^{-\tr
(y\xi)}-1\bigr)\mu(d\xi).
\end{eqnarray*}
Writing furthermore
\[
\wh p(t,y) = -p(t,-y),\qquad \wh q(t,y) = -q(t,-y),
\]
and similarly for $\phi$ and $\psi$, then, by \citet{cfmt},
\[
\mathbb{E}^{x} \bigl[e^{-\tr(yX_t)}\bigr]=e^{- \wh p(t,y)-\tr( \wh q(t,y)x)}
\]
for all $t\geq0$, $y,x\in S_d^+$, and by \citet{MayerhoferFinVar} the
exponents $(\wh p,\wh q)\dvtx \mathbb R_+\times S_d^+\rightarrow\mathbb
R_+\times S_d^+$ solve the system of generalized Riccati equations
\begin{subequations}
%
\begin{eqnarray}
\label{eq phi} \pd{} {t} \wh p(t,y)&=& \wh F\bigl(\wh q(t,y)\bigr),
\\
\label{eq psi} \pd{} {t} \wh q(t,y)&=&\wh R\bigl(\wh q(t,y)\bigr),
\end{eqnarray}
\end{subequations}
given initial data $\wh p(0,y)=0$, $\wh q(0,y)=y$.

Since $\mu$ is an $S_d^+$-valued measure, $\tr(\mu)$ is a
well-defined, nonnegative measure, naturally given by
\[
\tr(\mu) (A)=\tr\bigl(\mu(A)\bigr).
\]
Accordingly, the domain $\wh\cY:= - \cY$ is given by
%
\begin{equation}
\label{simplfied D matrix} \wh\cY=\biggl\{y\in S_d \Bigm|\int_{\|\xi\|\geq
1}e^{-\tr(y\xi)}
\bigl(m(d\xi)+\tr(\mu) (d\xi)\bigr)<\infty\biggr\}.
\end{equation}
The inclusion $\supseteq$ holds in view of the positive definiteness of
the measure $\mu$, while the inclusion $\subseteq$
follows from \citet{MayerhoferFinVar}, Lemma~3.3.
Similarly to the preceding section, we start with the following crucial
estimate: $I$ denotes the $d\times d$ unit matrix.

\begin{lem}\label{Lem:R_properties_complex_matrix}
Suppose that the diffusion coefficient satisfies $\alpha= I$ or
$\alpha=0$. Then there exists a locally Lipschitz continuous
function $g$ on $\wh\cY^\circ$ such that for all $u\in S(\wh\cY^\circ
)$ we have
%
\begin{equation}
\label{Item:Rnorm_matrix} \Re\bigl(\tr\bigl(\bar u \wh R(u) \bigr)\bigr)\leq g\bigl(\Re(u)
\bigr) \bigl(1+\|u\|^2\bigr).
\end{equation}
\end{lem}

\begin{pf}
As in the proof of Lemma~\ref{Lem:R_properties_complex} we start with
drift and big-jump parts. Clearly we have
%
\begin{equation}
\label{contributione1} \Re\tr\bigl(\bar uB^\top(u)\bigr)\leq G_1
\bigl(1+\|u\|\bigr)^2
\end{equation}
for some positive constant $G_1$. What concerns the big-jump parts, we
have
%
\begin{eqnarray}
\nonumber
\Re\tr\biggl(\bar u \biggl(\int_{\|\xi\|>1}
\bigl(e^{-\tr(u\xi)}-1 \bigr)\mu(d\xi) \biggr) \biggr) &\leq&\|u\|\tr(\mu)
\bigl(\bigl\{\xi\dvtx  \|\xi\|>1\bigr\}\bigr)
\\
\label{contributione2} &&{}+ \|u\| \int_{\|\xi\|>1}\bigl(e^{-\tr(\Re u\xi)}
\bigr)\tr(\mu) (d\xi)
\\
&\leq& g_2\bigl(\Re(u)\bigr) \bigl(1+\|u\|^2
\bigr)\label{contributione2b},
\end{eqnarray}
for some locally Lipschitz continuous function $g_2$. The integral
\eqref{contributione2} is finite, because $\Re(u)\in\wh\cY$ by
assumption. Here we have also used \citet{MayerhoferFinVar}, Lemma~3.3.
Note that we can set
\[
g_2(y):=\tr(\mu) \bigl(\bigl\{\xi\dvtx  \|\xi\|>1\bigr\}\bigr) + \int
_{\|\xi\|>1}\bigl(e^{-\tr
(y\xi)}\bigr)\tr(\mu) (d\xi).
\]

For $\alpha=0$ we set $g_3=0$. If $\alpha=I$, we involve Lemma~\ref
{involvemelemma} and obtain
%
\begin{equation}
\label{contributione3} \Re\bigl(\bar u u^2\bigr)\leq g_3\bigl(
\Re(u)\bigr)\|u\|^2,
\end{equation}
where $g_3(\cdot)=h(\cdot)=\pi(-\cdot)$.

It remains to estimate the small-jump part. Using again
\citet{MayerhoferFinVar}, Lemma~3.3, we have
\begin{eqnarray*}
&&\Re\tr\biggl(\bar u \int_{0<\|\xi\|\leq1}\bigl(e^{-\tr(u\xi)}-1
\bigr)\mu(d\xi)\biggr)
\\
&&\qquad=\Re\tr\biggl(\bar u \int_{0<\|\xi\|\leq1}\int_0^1
\tr(u\xi)e^{-s\tr
(u\xi)}\,ds\mu(d\xi)\biggr)
\\
&&\qquad\leq\|u\|^2 \int_{0<\|\xi\|\leq1}\int_0^1
e^{-s\tr(u\xi)} \|\xi\|\tr (\mu) (d\xi)
\\
&&\qquad\leq e^{\|\Re u\|}\|u\|^2 \int_{0<\|\xi\|\leq1}\|\xi\|
\tr(\mu) (d\xi)
\\
&&\qquad\leq g_4\bigl(\Re(u)\bigr) \bigl(1+\|u\|^2\bigr)
\end{eqnarray*}
with
\[
g_4(y):= e^{\|y\|} \int_{0<\|\xi\|\leq1}\|\xi\|
\tr(\mu) (d\xi).
\]
Summarizing the last estimate together with \eqref{contributione1},
\eqref{contributione2b} and \eqref{contributione3}
and setting
\[
g(y):=G_1+g_2(y)+g_3(y)+g_4(y)
\]
proves the assertion.
\end{pf}

We provide two further lemmas:

\begin{lem}\label{order regular D matrix}
If $y\in\wh\cY^\circ$, and $z\in S_d$ such that $z\succeq y$, then we
also have $z\in\wh\cY^\circ$.
\end{lem}

\begin{pf}
Using \eqref{simplfied D matrix} we infer the existence of some
$\varepsilon>0$ such that
for all $w\in B_{\varepsilon}(0)=\{w\in S_d\mid\|w\|<\varepsilon\}$,
we have
\[
\int_{\|\xi\|\geq1}e^{-\tr((y+w)\xi)}\bigl(m(d\xi)+\tr(\mu) (d\xi)
\bigr)<\infty.
\]
The assumption of the lemma implies that $\langle z,\xi\rangle\geq
\langle y,\xi\rangle$
for all $\xi\in S_d^+$. Furthermore, $m$ and $\tr(\mu)$ are supported
on $S_d^+$. Therefore we have
\begin{eqnarray*}
&&\int_{\|\xi\|\geq1}e^{-\tr((z+w)\xi)}\bigl(m(d\xi)+\tr(\mu) (d\xi)
\bigr)
\\
&&\qquad\leq\int_{\|\xi\|\geq1}e^{-\tr((y+w)\xi)}\bigl(m(d\xi)+\tr(\mu) (d\xi
)\bigr)<\infty
\end{eqnarray*}
for all $w\in B_{\varepsilon}(0)$, which in view of
\eqref{simplfied D matrix} proves that $z\in\wh\cY^\circ$.
\end{pf}

\begin{lem}\label{qmi R matrix}
$\wh R$ is quasimonotone increasing (with respect to $S_d^+$) on
$\wh\cY^\circ$.
\end{lem}

\begin{pf}
The proof is analogous to the one of \citet{cfmt}, Lemma~5.1, which
states quasimonotonicity of $\wh R$ on $S_d^+$.
\end{pf}

We are now prepared to prove Proposition~\ref{c blow up after r} for
$D=S_d^+$, $d\geq2$:

\begin{pf}
According to \citet{cfmt}, Theorem~4.14, for any affine process $X$
(with
diffusion coefficient $\alpha$) there exists a linear automorphism $g$
of $S_d^+$ such that the affine process $Y=g(X)$
has diffusion coefficient $\widetilde{\alpha}=\diag(I_r,0)$, where
$I_r$ is the $r\times r$ unit matrix, and $r=\rank(\alpha)$. According
to our assumption $r=0$ or $r=d$ (see Remark~\ref{rem assumption
matrix}), and linear transformations do not
affect the blow-up relation (between the real and complex-valued
solutions) we are about to prove here. Hence we may without loss of
generality assume that
$\alpha=0$ or $\alpha=I$.

For any $u\in S(\wh\cY^\circ)$ we write $y=\Re(u)$. The
quasimonotonicity of $\wh R$ (Lem\-ma~\ref{qmi R matrix}) allows us to
apply
the multivariate comparison result by \citet{Volkmann1973}, and we
conclude that for $t<T_+(u)\wedge T_+(y)$, we have
$\Re\wh\psi(t,u)\succeq\wh q(t,y)$. In view of Lemma~\ref{order
regular D
matrix} we only need to show that $t\mapsto\|\wh\psi(t,u)\|$ does not
explode before $t\mapsto\|\wh q(t,y)\|$. By Lemma~\ref{Lem:R_properties_complex_matrix}, there exists a continuous
function $g$ such that
\[
\Re\tr\bigl(\bar u \wh R(u)\bigr)\leq g\bigl(\Re(u)\bigr) \bigl(1+\|u
\|^2\bigr), \qquad u\in S\bigl(\wh\cY^\circ\bigr).
\]
Hence, we have for all $t<T_+(u)\wedge T_+(y)$,
\[
\pd{} {t} \bigl(\bigl\|\wh\psi(t,u)\bigr\|^2\bigr)=2\Re\tr\bigl(\overline{\wh
\psi(t,u)}\wh R\bigl(\wh\psi(t,u)\bigr)\bigr)\leq g\bigl(\Re\bigl(\wh\psi(t,u)
\bigr)\bigr) \bigl(1+\bigl\|\wh\psi(t,u)\bigr\|^2\bigr),
\]
and by Gronwall's inequality, we obtain
\[
\bigl\|\wh\psi(t,u)\bigr\|\leq\bigl(1+\|u\|^2\bigr) \int_0^t
g(s)e^{\int_0^s
g(\xi)\,d\xi}\,ds.
\]
Hence we have shown that $T_+(u)\geq T_+(y)$.
\end{pf}

\subsection{Proof of Theorem \texorpdfstring{\protect\ref{thmm:main_complex}}{2.26}}
The first part of Theorem~\ref{thmm:main_complex} is proved in
Proposition~\ref{c blow up after r}. For the proof of the second
part, the validity of the complex transform formula \eqref{eq:
complex transform formula}, we utilize the concept of analytic
continuation.

\begin{pf*}{Proof of Theorem~\ref{thmm:main_complex}}
Consider the set
\[
U:=\bigl\{y\in\cY^{\circ}\mid T_+(y)>T\bigr\}.
\]
By assumption $U$ is nonempty, and from the standard existence and
uniqueness theorem for
ODEs it follows that $U$ is open. Next, we show that $U$ is convex. For
$y_1, y_2 \in U$ it follows from Theorem~\ref{thmm:main_results}(b) on
real moments that $\mathbb{E}^{x}[e^{\langle{y_1},{X_T}\rangle}] <
\infty$ and $\mathbb{E}^{x}[e^{\langle{y_2},{X_T}\rangle}] < \infty
$ for all $x \in D$. Let $\lambda\in
[0,1]$ and set $y_\lambda= \lambda q_1 + (1 - \lambda) q_2$. By H\"
older's inequality
\[
\mathbb{E}^{x}\bigl[e^{\langle{y_\lambda},{X_T}\rangle}\bigr] \le\mathbb
{E}^{x}\bigl[e^{\langle{y_1},{X_T}\rangle}\bigr]^\lambda\cdot\mathbb
{E}^{x}\bigl[e^{\langle{y_2},{X_T}\rangle}\bigr]^{(1 - \lambda)} < \infty
\]
for all $x \in D$ and we conclude, using Theorem~\ref
{thmm:main_results}(a) that $y_\lambda\in U$ and hence that $U$ is convex.
Now set $U':=S(U)\subset\mathbb C^d$. From the properties of $U$ we
conclude that $U'$ is nonempty, open and connected. By
Proposition~\ref
{c blow up
after r} we have $T_+(u')>T_+(\Re u')$ for all $u'\in U'$. Furthermore, since
$u\mapsto R(u)$ and $u\mapsto F(u)$ are complex analytic on
$\cY^\circ$, we have by \citet{dieu69}, Theorem~10.8.2, that the
function
\[
M(u):= e^{\phi(t,u) + \langle{\psi(t,u)},{x}\rangle}
\]
is complex analytic on $U'$ for all $t \le T$. Furthermore, by
Theorem~\ref{thmm:main_results} and by Remark~\ref{rem comparison},
we have
that
\[
M(y)=e^{\phi(t,y) + \langle{\psi(t,y)},{x}\rangle}=\mathbb
E^x\bigl[e^{\langle{y},{X_t}\rangle}\bigr],\qquad y\in
U, t \le T.
\]
We conclude that the function $\Phi(u)\dvtx U'\rightarrow\mathbb C\dvtx u\mapsto\mathbb E^x[e^{\langle{u},{X_t}\rangle}]$ is an analytic
function, which
coincides with $M(u)$ on the nonempty open subset $U \subset U'$.
Hence by
the principle of analytic continuation\footnote{Here we use that $U'$
is open and connected.} [cf.~\citet{dieu69}, (9.4.4)] $\mathbb
E^x[e^{\langle{u},{X_t}\rangle}]=M(u)$ on all of
$U'$, and the proof is complete.
\end{pf*}

\section*{Acknowledgments}
We thank Chulmin Kang for valuable comments on Section~\ref{sec4.1}.

%


\printaddresses

\end{document}